\documentclass[11pt]{amsart}
\usepackage{amssymb,amsmath, amsthm}
\usepackage{amsfonts,url, hyperref}
\DeclareMathOperator{\diag}{diag}
\DeclareMathOperator{\Aut}{Aut}
\DeclareMathOperator{\Char}{char}
\DeclareMathOperator{\Hom}{Hom}
\DeclareMathOperator{\Ext}{Ext}
 
\DeclareMathOperator{\lm}{lm}

\DeclareMathOperator{\lc}{lc}

\DeclareMathOperator{\NF}{NF}

\DeclareMathOperator{\trdeg}{tr.deg}
\DeclareMathOperator{\Gr}{Gr}

\DeclareMathOperator{\gldim}{gl.dim }

\DeclareMathOperator{\gkdim}{GKdim}

\newtheorem{theorem}{Theorem}[section]
\newtheorem{lemma}[theorem]{Lemma}
\newtheorem{prop}[theorem]{Proposition}
\newtheorem{proposition}[theorem]{Proposition}
\theoremstyle{definition} 
\newtheorem{definition}[theorem]{Definition}
\newtheorem{remark}[theorem]{Remark}
\newtheorem{example}[theorem]{Example}

\newtheorem{cor}[theorem]{Corollary}
\begin{document}
\newcommand{\DOT}{\setlength{\unitlength}{1pt}\begin{picture}(2.5,2)
               (1,1)\put(2,3.5){\circle*{3}}\end{picture}}
\newcommand{\HHD}{{\rm HH}^{\DOT}}
\newcommand{\cH}{{\mathcal H}}
\newcommand{\cA}{{\mathcal A}}
\newcommand{\qdha}{quantum Drinfeld Hecke algebra}
\newcommand{\Lra}{\Leftrightarrow}
\newcommand{\lra}{\longrightarrow} 
\newcommand{\Ra}{\Rightarrow}
\newcommand{\F}{{\mathbb F}}
\newcommand{\FF}{{\mathcal F}}
\newcommand{\N}{{\mathbb N}}
\newcommand{\K}{{\mathbb K}}
\newcommand{\GL}{\mathbb{GL}}
\newcommand{\SL}{\mathbb{SL}}
\newcommand{\KG}{{\K}G}
\newcommand{\C}{{\mathbb C}}
\newcommand{\R}{{\mathbb R}}
\newcommand{\Z}{{\mathbb Z}}
\newcommand{\Q}{{\mathbb Q}}
\newcommand{\E}{{\mathcal E}}
\newcommand{\nL}{{\mathcal L}}
\newcommand{\nNull}{{\mathrm {\mathbf 0}}}
\newcommand{\kg}{{\mathcal G}}
\newcommand{\ki}{{\mathcal I}}
\newcommand{\kv}{{\mathcal V}}
\newcommand{\ox}{{\overline{x}}}
\newcommand{\D}[1] {\partial_{#1}}
\renewcommand{\d}{{\partial}}
\newcommand{\dx}{{\partial_x}}
\newcommand{\dy}{{\partial_y}}
\newcommand{\mb}[1] {\mathbf{#1}}
\newcommand{\mT}{{\mathcal T}}
\newcommand{\kD}{{\mathcal D}}
\newcommand{\I}{{\mathcal I}}
\newcommand{\J}{{\mathcal J}}
\newcommand{\kb}{{\mathcal B}}
\newcommand{\A}{\ensuremath{\mathcal{A}}}
\newcommand{\B}{\ensuremath{\mathcal{B}}}
\newcommand{\g}{{\ensuremath{\mathfrak{g}}}}
\newcommand{\f}{{\ensuremath{\mathfrak{f}}}}
\newcommand{\gl}{{\mathfrak{ll}}}
\renewcommand{\t}{{\mathfrak{t}}}
\newcommand{\fI}{{\mathfrak{I}}}
\newcommand{\ra}{\rightarrow}
\newcommand{\GZS}{\operatorname{GZSupp}}
\renewcommand{\L}{{\mathcal{L}}}
\newcommand{\zsupp}{\ensuremath{\operatorname{Supp}}}
\newcommand{\n}{\ensuremath{\mathfrak{n}}}
\newcommand{\m}{\ensuremath{\mathfrak{m}}}
\newcommand{\ann}{\operatorname{Ann}}
\newcommand{\spec}{\operatorname{Spec}}
\newcommand{\Ann}{\operatorname{Ann}}
\newcommand{\rank}{\operatorname{rank}}
\newcommand{\syz}{\operatorname{Syz}}
\newcommand{\Syz}{\operatorname{Syz}}
\newcommand{\LeftSyz}{\operatorname{LeftSyz}}
\newcommand{\id}{\mathbb{I}\text{d}}
\newcommand{\ot}{\otimes}
\newcommand{\X}{\langle X \rangle}
\newcommand{\KX}{\K\X}
\newcommand{\algebraitem}[6] {\rule[-1mm]{0mm}{5mm}#1 & #2 & #3\\ \hline
\multicolumn{3}{|l|}{\parbox[t]{120mm}{\setlength{\tabcolsep}{0mm}
\begin{tabular}{l@{\quad}rll}
\rule{0mm}{5mm}Isomorphism: & $X$ & $\to\;$ & #4\\
\rule[-1mm]{0mm}{0mm}& $Y$ & $\to$ & #5
\end{tabular}}}\\
\multicolumn{3}{|l|}{\parbox[t]{120mm}{#6}}\\ \hline}
\title[Quantum Drinfeld Hecke Algebras]{Quantum Drinfeld
Hecke Algebras}
\date{October 2013.}
\author{Anne V.\ Shepler}
\address{Anne V.\ Shepler, Department of Mathematics, University of North Texas,
Denton, Texas 76203, USA}
\email{ashepler@unt.edu}
\author{Viktor Levandovskyy}
\address{Viktor Levandovskyy, Lehrstuhl D f\"ur Mathematik,
RWTH Aachen University,
Templergraben 64,
D-52062 Aachen, Germany}
\email{levandov@math.rwth-aachen.de} 

\thanks{Work of the second author was partially supported by National Science
Foundation research grants
\#DMS-0800951 and \#DMS-1101177
and a research fellowship from the 
Alexander von Humboldt Foundation.}
\begin{abstract}
We consider finite groups acting on 
quantum (or skew) polynomial rings.  Deformations of the 
semidirect product of the quantum polynomial ring with the acting group
extend symplectic reflection algebras and graded Hecke algebras
to the quantum setting over a field
of arbitrary characteristic.  
We give necessary and sufficient conditions for such algebras to satisfy a 
Poincar\'e-Birkhoff-Witt property using the theory of noncommutative 
Gr\"obner bases.
We include applications to the case of abelian groups
and the case of groups acting on coordinate rings of quantum planes.
In addition, we classify graded automorphisms of the coordinate ring of quantum 3-space.  In characteristic zero, Hochschild cohomology
gives an elegant description of the Poincar\'e-Birkhoff-Witt conditions.
\end{abstract}
\maketitle
\section{Introduction}
Drinfeld Hecke algebras arise in a variety of settings: 
for example, as symplectic reflection algebras, 
rational Cherednik algebras, and Lusztig's graded version
of the affine Hecke algebra.
These algebras (also known as {\em graded Hecke algebras})
are natural deformations
of the skew group algebra 
(the semi-direct product algebra)
formed by 
a finite group $G$ acting 
on a polynomial ring over some vector space $V$.
They reflect the geometry of orbifold theory
by serving as a noncommutative substitute for the 
coordinate ring 
(the ring of invariant polynomials $S(V)^G$)
of the orbifold $V/G$
(see Etingof and Ginzburg~\cite{EtingofGinzburg}).  
These algebras were also used to prove a version of the $n!$ conjecture
for Weyl groups (see Gordon~\cite{Gordon}).

In this article, we explore analogous deformations of a finite
group acting on a quantum polynomial algebra
over a field of arbitrary characteristic.
Let $V$ be a finite dimensional vector space
over a field $\K$.
The {\em quantum polynomial algebra} 
$S_Q(V)$ 
of $V$ (also called the 
{\em skew polynomial ring}, or the
{\em coordinate ring of multiparameter quantum affine space})
is the
associative $\K$-algebra generated by  
a $\K$-basis $\{v_1,\ldots, v_n\}$ of $V$
subject to the relations
$v_j v_i = q_{ij} v_i v_j$ for $i < j $ for some 
(quantum) parameters $q_{ij}$ in $\K^*$:
$$
S_Q(V):=\K\left<v_1, \ldots, v_n\right>/
\left<v_j v_i - q_{ij} v_i v_j:1\leq i < j \leq n\right>\, .
$$

We augment the quantum polynomial algebra by a finite
group $G$ acting linearly on the vector space $V$.
We introduce relations on the natural semi-direct
product algebra $T(V)\rtimes G$ (for $T(V)$ the tensor algebra
of $V$)
which set $q$-commutators
of vectors in $V$ to elements in the group algebra.
We call the resulting $\K$-algebra a {\em \qdha}
if it satisfies a Poincar\'e-Birkhoff-Witt
(PBW) property (see Definition~\ref{defn}).  

We appeal to the theory of noncommutative 
Gr\"obner bases to  
investigate PBW properties.
Explorations of related algebras often
use Bergman's Diamond Lemma~\cite{Bergman},
a cornerstone of noncommutative Gr\"obner bases theory.
We use Gr\"obner bases theory here as a rigorous 
and elegant refinement of Bergman's ideas. This refinement is 
well-suited to investigating  
PBW-like 
properties in a variety of settings.
Indeed, the constructive nature of Gr\"obner bases
theory often verifies a PBW-like  property
by explicitly giving a PBW-like basis; 
the theory also illuminates the
failure of such properties to hold by supplying natural substitutes for
PBW-like bases.
Note that Gr\"obner bases theory emphasizes a fixed total well ordering
on monomials,
providing an expedient approach to non-noetherian algebras.
Indeed, 
a Gr\"obner basis for a generating set of relations defining an algebra 
may be finite for one choice of monomial ordering but infinite for another choice;
see Example~\ref{exPBW}.
Moreover, Gr\"obner bases theory is algorithmic with several available
implementations, providing computational aid to
algebraic questions (on, e.g., ideal membership,
kernels of algebra and module homomorphisms, and free and projective
resolutions).

Although \qdha   
s extend symplectic reflection algebras and graded Hecke algebras
to the setting of quantum polynomial rings,
our analysis requires tools previously unused 
in investigating the nonquantum setting.
Since we are working over a field of arbitrary characteristic,
many methods from the traditional theory of graded Hecke algebras
no longer apply.
(Note that the original proof 
of the technique of Braverman and Gaitsgory~\cite{BravermanGaitsgory}
does not automatically apply in our setting, as the group algebra $\K G$
may fail to be semi-simple;
see~\cite{PBWTheorem} for an adaptation of the ideas of
Braverman and Gaitsgory for arbitrary group algebras, including the modular
case when the characteristic of the field $\K$ divides the order of the acting group $G$.)
The set of quantum parameters also prevents us from regarding the
algebra parameters as linear functions giving a wide class of uniform
relations (see Remark~\ref{Bilinear Inextendability}), 
and thus we demote traditional linear algebra
in favor of the analysis using noncommutative Gr\"obner bases.

After giving definitions (and examples) in Section~\ref{DefinitionsSection},
we show that every \qdha\ defines a quantum polynomial algebra
upon which the  group acts by automorphisms
in Section~\ref{QuantumPolynomialRingsSection}.
Tools from the theory of noncommutative
Gr\"obner bases theory are given in Sections~\ref{NoncommutativeTheorySection}
and~\ref{GBSection}. 
In Section~\ref{GBPBWBasisSection},
we recall how a Gr\"obner basis may be used to find a 
monomial $\K$-basis for any quotient of a free algebra
by one of its ideals.  
We also discuss general quotient
algebras and associated graded algebra.  (Some 
elementary algebraic properties of \qdha s are also observed 
in this section.) 
We apply this theory in Section~\ref{ConditionsOnParametersSection}
to prove necessary and sufficient conditions
for a factor algebra to define a \qdha.
In Section~\ref{AbelianSection}, we describe
all \qdha s arising from an abelian group (acting diagonally).
We relate the Poincar\'e-Birkhoff-Witt  
condition for \qdha s to results in Hochschild
cohomology and deformation theory by Naidu and Witherspoon~\cite{NW}
in Section~\ref{HochschildCohomologySection}.

In Section~\ref{QuantumPlanesSection}, we discuss
groups which act as automorphisms on the coordinate ring of a quantum plane
and classify all \qdha s in two dimensions.
We describe the automorphism group of the coordinate ring of quantum $3$-space
in Section~\ref{QuantumSpaceSection}.  (We discuss the cases when quantum
parameters are roots-of-unity explicitly.)
Lastly, in Section~\ref{ExtensiveExampleSection},
we demonstrate how to determine
the complete set of \qdha s associated to one fixed (nonabelian) group
with a robust example. 

\vspace{2ex}

\section{Quantum Drinfeld Hecke Algebras}\label{DefinitionsSection}
Let $Q = (q_{ij} \mid 1\leq i,j \leq n)$ 
be a collection
of arbitrary nonzero scalars in $\K$
and consider a finite group $G \subset \GL(V)$.
Let $\{t_g \mid g \in G\}$ be a basis of 
the group algebra $\KG$ and $\{v_1, \ldots, v_n \}$ a $\K$-basis of $V$. 
Define an associative $\K$-algebra $\cH_{Q,\kappa}$ generated by 
\[
\{v_1, \ldots, v_n \} \cup  \{t_g \mid g\in G\} 
\]
subject to the following relations:
\begin{enumerate}
\item[(a)] $t_g t_h = t_{gh}$ for all  
$g,h$ in $G$,
\item[(b)]
 $t_g v = g(v) t_g$  for all $g$ in $G$ and $v$ in $V$,
\item[(c)] 
$v_j v_i = q_{ij}\  v_i v_j + \kappa(v_i, v_j)$ \ 
 for $1\leq i,j\leq n$\ ,
\end{enumerate}
where each parameter
$\kappa(v_i,v_j)$ lies in $\KG$.
Write $$\kappa(v_i, v_j)=\sum_{g\in G} \kappa_g(v_i, v_j) t_g$$ 
for $\kappa_g(v_i,v_j)$ in $\K$.
We identify the identity $e$ of $G$ and $t_e$ of $\KG$ with $1$ in $\K$ 
throughout this article
and we set $G^*:=G\setminus\{1\}$. 
We assume $0$ lies in $\N$ and take all tensor products over $\K$.

\vspace{2ex}

\begin{definition}\label{defn}
We call $\cH_{Q,\kappa}$ a
{\em quantum Drinfeld Hecke algebra} 
if 
$$B=\{v_1^{\alpha_1} \ldots v_n^{\alpha_n} t_g \mid \alpha_i\in\N, \ g\in G\}$$
is a $\K$-basis for $\cH_{Q,\kappa}$.
We call $B$ the {\em standard PBW basis} in this case
and its elements {\em quasi-standard monomials}. 
\end{definition}
\vspace{2ex}

One might alternatively call such algebras ``quantum graded Hecke algebras''
or ``skew Drinfeld Hecke algebras''.
We use the phrase ``PBW basis'' in analogy with a
Poincar\'e-Birkhoff-Witt basis for universal enveloping algebras 
of Lie algebras.
Note that $\cH_{Q,\kappa}$
is a \qdha\  if and only if its associated graded algebra
is isomorphic to 
a skew group algebra $S_Q(V)\# G$
(see Sections~\ref{NoncommutativeTheorySection} 
and~\ref{GBPBWBasisSection}).

The braided Cherednik algebras of
Bazlov and Berenstein~\cite{BazlovBerenstein} are 
special cases of \qdha s.
If we set each $q_{ij}=1$ in the above construction
of $\cH_{Q,\kappa}$ (and work over a field $\K$ of characteristic
zero), we recover the classical (non-quantum) theory of
{\em graded Hecke algebras}, also called {\em Drinfeld Hecke algebras}
(see~\cite{Griffeth}, for example),
which include symplectic reflection algebras and rational
Cherednik algebras.
These algebras were first defined by Drinfeld~\cite{Drinfeld}
for arbitrary finite groups $G$.  They were independently 
discovered and explored
by Lusztig around the same time (see~\cite{Lusztig88, Lusztig89}) 
as graded versions of the affine Hecke algebra
in the special case that $G$ is a Weyl group.
(See~\cite{RamShepler} for basic properties of these algebras
and an argument that Lusztig's algebras
can be realized using Drinfeld's construction.)
Etingof and Ginzburg~\cite{EtingofGinzburg} later rediscovered
these algebras (from a viewpoint of symplectic geometry and
orbifold theory)
for $G$ acting symplectically. 
We give some other examples with fixed quantum system of parameters.

\vspace{3ex}

\begin{definition}
\label{QSP}
A matrix 
$Q = (q_{ij} \mid 1\leq i,j \leq n)$ 
with entries in $\K^*$
is a {\em quantum system of parameters} if 
$q_{ij} = q_{ji}^{-1}$ and $q_{ii}=1$ for any $i,j$.
\end{definition}

\vspace{1ex}

\begin{example}
Set $\kappa\equiv 0$, $G=1$, and let
$Q$ be a quantum system of parameters.
Then the factor algebra $\cH_{Q,\kappa}$
is just the quantum polynomial algebra $S_Q(V)$.
\end{example}

\vspace{1ex}

\begin{example}\label{quantumexterioralgebra}
Again, 
let $\kappa\equiv 0$, $G=1$, and let
$Q$ be a quantum system of parameters.
Assume that $\text{char }\K\neq 2$
and set $-Q=(-q_{ij}\mid 1\leq i,j\leq n)$.
Then the factor algebra $\cH_{-Q,\kappa}$
coincides with the 
{\em quantum exterior algebra}
$\bigwedge_Q(V)$ of quantum affine space (corresponding
to the quantum polynomial algebra $S_Q(V)$)
generated over $\K$ by all products
$v_{i_1}\wedge \cdots\wedge v_{i_k}$ (for $1\leq k \leq n$)
with multiplication
$$v_i\wedge v_j = -q_{ji} \ v_j \wedge v_i\ .$$
Although $S_Q(V)$ has the standard PBW basis 
(e.g., see~\cite[Example~5.1]{BGV}
or Proposition~\ref{KQV}), the algebra $\bigwedge_Q(V)$ does not
(as each $v_i\wedge v_i=0$). 
(In fact, it is easy to see that any \qdha\ with $\kappa\equiv 0$ 
is a quantum polynomial algebra, see 
Proposition~\ref{skewpolyalg}.)
\end{example}

\vspace{2ex}

\begin{example}
Let $q,\omega$ be roots of unity in $\K$ and let $G$ 
be the subgroup of $\GL_4(\K)$ generated by the diagonal matrix
$$
h:= \text{diag}
(q^2, \omega,\, \omega^{-1}, q^{-2}) \, .
$$
The $\K$-algebra $\cH$ generated by $v_1, v_2, v_3, v_4$ and $t_h$ 
with relations
$$
\begin{aligned}
t_g v_i &= g(v_i) t_g &&\text{ for } 1\leq i\leq n \text{ and } g \text{ in } G,\\
v_i v_j &= q\, v_j v_i &&\text{ for } (i,j)\neq (2,3),
\text{ and }\\
v_2 v_3&=q\, v_3 v_2 + t_h \ 
\end{aligned}
$$
is a \qdha.
\end{example}

\quad

In Section~\ref{AbelianSection}, we describe
all \qdha s arising from abelian groups acting diagonally.
We classify all 2-dimensional
\qdha s in Section~\ref{QuantumPlanesSection}.
Section~\ref{ExtensiveExampleSection}
gives examples of \qdha s arising from a nondiagonal group action.

\vspace{2ex}

\begin{remark}\label{Bilinear Inextendability}
(Bilinear Inextendability)
We define parameters $q,\kappa$ just on pairs of basis elements $v_i,v_j$,
but we could (artificially) extend to functions 
$q: V\times V \rightarrow \K$ and $\kappa: V\times V \rightarrow \KG$.
This approach is generally not useful for constructing factor algebras
like those examined here (although it is helpful in translating
results to the setting of cohomology, see 
Section~\ref{HochschildCohomologySection}).

For example, 
suppose we were to extend relation (3) defining the algebra $\cH_{Q,\kappa}$
to all pairs $v,w$ in $V$ using
a {\em bilinear} function $\kappa$ and some function $q:V\times V \rightarrow \K$.
Then in $\cH_{Q,\kappa}$, for any distinct $i,j,k$,
$$(v_i+v_j)v_k=v_iv_k+v_jv_k
=q(v_k,v_i) v_kv_i+ q(v_k,v_j) v_kv_j +\kappa(v_k,v_i+v_j)\ 
 $$
on one hand, while 
$$
 (v_i+v_j)v_k=q(v_k,v_i+v_j) v_k(v_i+v_j)+\kappa(v_k,v_i+v_j) \ 
$$
on the other hand, forcing
 $$
 q(v_k,v_i)v_kv_i+q(v_k,v_j)v_kv_j=q(v_k,v_i+v_j)v_kv_i+q(v_k,v_i+v_j)v_kv_j.
 $$
If $\cH_{Q,\kappa}$ has the standard PBW basis, we may equate 
coefficients:
 $$
 q(v_k,v_i)=q(v_k,v_i+v_j)=q(v_k,v_j).
 $$
 This forces $q$ constant on basis vectors, i.e., 
 $q_{ij}=c$ for all $i,j$, for fixed $c$ in $\K$.  
Note that $q$ bilinear would generally imply  
that $q$ is the zero function.
\end{remark}
\section{Quantum Polynomial Algebras, Quantum Determinants,
and Skew Group Algebras}\label{QuantumPolynomialRingsSection}

We show in this section that every quantum Drinfeld Hecke algebra
defines a quantum polynomial algebra
carrying an action of the group
by automorphisms.
We first give an easy lemma describing automorphisms
of quantum polynomial algebras in terms of quantum minor determinants.
Any automorphism $h$ of the quantum exterior algebra
$\bigwedge_Q(V)$ will act on the top
degree piece $\K$-span$\{v_1\wedge\cdots\wedge v_n\}$
by a scalar $\det_Q(h)$ which one might call the
{\em quantum determinant} of $h$.
We extend this idea:
If a $2\times 2$ matrix with entries $a,b,c,d$ in $K$ has determinant
$ad-bc$, then we define its quantum determinant to be $ad-qbc$ 
where $q$ is the quantum parameter of a $2$-dimensional quantum polynomial ring.
We define a quantum minor analogously.
\vspace{1ex}

\begin{definition}
For a linear transformation $h$ acting on $V$ via
$h(v_j)=\sum_i h_i^j v_i$, we define
the {\em quantum $(i,j,k,l)$-minor determinant} of $h$ as
$$
{\det}_{ijkl}(h):=h^i_{k}h^j_{\ell} - q_{ij} h^i_{\ell} h^j_{k} .
$$
\end{definition}
%
\vspace{2ex}

\begin{lemma}\label{auts}
A transformation $h$ in $\GL(V)$ acts as an automorphism
on the quantum polynomial algebra (with quantum system of parameters $Q$)
$$
S_Q(V):=\K\left<v_1, \ldots, v_n\right>/
\left<v_j v_i = q_{ij} v_i v_j:1\leq i, j \leq n\right>
$$ 
if and only if 
$$
{\det}_{ijk \ell}(h)= - q_{\ell k}\ {\det}_{ij\ell k}(h)\ 
\quad\text{ for all } 1\leq i,j,k,\ell \leq n\ .
$$
\end{lemma}
\begin{proof}
We write
$
h(v_j)h(v_i)-q_{ij}h(v_i)h(v_j)=
\sum_{k,\ell} {\det}_{ij\ell k}(h)\ v_k v_{\ell}
$
and express as a sum of standard monomials:
$$
\begin{aligned}
\sum_{k\leq \ell} {\det}_{ij\ell k}(h)\ &  v_k v_{\ell}+
 \sum_{k> \ell} {\det}_{ij\ell k}(h)\ v_k v_{\ell}\ = \\
\sum_{k<\ell} &({\det}_{ij\ell k}(h) 
 + {\det}_{ijk\ell}(h) \ q_{k\ell})\ v_{k} v_{\ell}
+\sum_k {\det}_{ijkk}(h)\ v_k v_k\, .\\
\end{aligned}
$$
Since the set of standard monomials $\{v_1^{\alpha_1}\cdots v_n^{\alpha_n}:\alpha_i\in \N\}$ is a $\K$-basis of $S_Q(V)$
(see, e.g.,~\cite[Example~5.1]{BGV} or Corollary~\ref{KQV}),
the last expression vanishes in $S_Q(V)$ exactly when
the coefficient of each $v_k v_\ell$ (for $k<\ell$)
and of each $v_k^2$ is zero, yielding the result.  
\end{proof}
As an easy consequence (needed later), we observe
the following corollary.
\vspace{1ex}

\begin{cor}
\label{transpose}
A matrix $h$ in $\GL(V)$ 
acts as an automorphism on $S_Q(V)$ 
if and only if its transpose 
acts as an automorphism on $S_Q(V)$.
\end{cor}

\vspace{1ex}

\begin{definition}
We say that a parameter $\kappa$ is a {\em quantum $2$-form}
if $\kappa$ extends to an element of 
$\Hom_{\K}(\bigwedge_Q(V),\K G)$, i.e.,
each $\kappa_g$ defines an element of 
$${\textstyle \bigl(\bigwedge_Q(V)\bigr)^*\cong \bigwedge_{Q^{-1}} (V^*)}$$
where $Q^{-1}=(q_{ij}^{-1}:1\leq i,j\leq n)$.
In other words, $\kappa$ is a quantum $2$-form exactly when
$\kappa(v_i, v_i)=0$ and
$\kappa(v_j, v_i)=-q_{ij}^{-1}\kappa(v_i,v_j)$ for
all $i,j$.
\end{definition}

\vspace{3ex}

A PBW property on $\cH_{Q,\kappa}$ implies an underlying
quantum polynomial algebra.
\begin{prop}\label{skewpolyalg}
Let $\cH_{Q,\kappa}$ be a quantum Drinfeld Hecke algebra.
Then
\begin{itemize}
\item 
the parameter $\kappa$ is a quantum 2-form,
\item the matrix $Q$ is a quantum system of parameters, and 
\item the group $G$ acts upon 
the quantum polynomial algebra $S_Q(V)$ by automorphisms.
\end{itemize}
\end{prop}
\begin{proof}
Since $\cH_{Q,\kappa}$ exhibits the standard PBW basis,
each $q_{ii}=1$ and each $\kappa(v_i,v_i) = 0$
as $v_i^2 = q_{ii} v_i^2 + \kappa(v_i,v_i)$.   In  fact, 
for all $i$ and $j$,
$$
\begin{aligned}
v_jv_i&=q_{ij}v_i v_j+\kappa(v_i,v_j)
=q_{ij}\big(q_{ji} v_jv_i+\kappa(v_j,v_i)\big)+\kappa(v_i,v_j)\\
&=q_{ij}q_{ji} v_jv_i+q_{ij}\kappa(v_j,v_i)+\kappa(v_i,v_j)\ ,
\end{aligned}
$$
and hence
$
q_{ij}\neq 0,\ q_{ij}=q_{ji}^{-1}, \text{ and }
\kappa(v_j, v_i)=-q_{ij}^{-1} \kappa(v_i, v_j). 
$
Thus
$\kappa$ is a quantum 2-form and
$Q$ defines a quantum system of parameters.

Additionally, for all $h$ in $G$ and $i\neq j$,
\begin{equation}\label{expanding}
\begin{aligned}
0
&
=
(t_h v_j) v_i t_{h^{-1}} 
- 
t_h (v_j v_i) t_{h^{-1}} 
\\
&=
h(v_j) (t_h v_i) t_{h^{-1}}
-
t_h\biggl(
q_{ij}v_iv_j+\sum_{g\in G}\kappa_g(v_i,v_j)t_g\biggr)t_{h^{-1}}
\\
&=
h(v_j)h(v_i)
-
q_{ij}h(v_i)h(v_j)
-
\sum_{g\in G}\kappa_g(v_i,v_j)t_{hgh^{-1}}\\
&=
\sum_{k,\ell} {\det}_{ij\ell k}(h)\  v_k v_{\ell}
-\sum_{g\in G}\kappa_{h^{-1}gh}(v_i,v_j)t_{g}\ .\\
\end{aligned}
\end{equation}
We separate the sum of
${\det}_{ij\ell k}(h)\ v_k v_{\ell}$
over $k>\ell$ and exchange $v_\ell$ and $v_k$ to
express Equation~\ref{expanding} using only quasi-standard 
monomials:
\begin{equation}\label{thirdformalrelation}
\begin{aligned}
0= \ \ \ \ \sum_{k < \ell} 
& \biggl(
q_{k\ell}\ {\det}_{ijk\ell}(h) + {\det}_{ij\ell k}(h)
\biggr)\ v_{k} v_{\ell}  + 
\sum_{k}{\det}_{ijkk}(h)\ v_{k}^2 \\
& -
\sum_{g\in G} \Bigg(\sum_{k < \ell}\ {\det}_{ijk\ell}(h)\   
\kappa_g(v_k,v_\ell)- \kappa_{h^{-1}gh}(v_i,v_j)\Bigg)\ t_{g} \ .
\end{aligned}
\end{equation}
Since $\cH_{Q,\kappa}$ has the standard PBW basis, the
coefficient of each monomial $v_k v_\ell$
and $v_k^2$ in the above sum must be zero.
Lemma~\ref{auts} then implies
that
the action of $G$ on $V$ extends to an action of 
$G$ on $S_Q(V)$ by automorphisms.
\end{proof}

Recall that a matrix in $\GL_n(\K)$ is 
{\em monomial} if each column and each row has exactly one
nonzero entry.  A subgroup $G\leq \GL(V)$
is called {monomial} with respect to a fixed basis of $V$
if it acts by monomial matrices.
\begin{cor}\label{monomialgroup}
Suppose $\cH_{Q,\kappa}$ is a \qdha.
If each $q_{ij}\neq 1$ with $i\neq j$, then $G$ is a monomial group.
\end{cor}
\begin{proof}
Fix $h$ in $G$ and write $h(v_a)=\sum_{b}h_b^a v_b$ for each $1\leq a\leq n$.
The previous proposition and lemma imply that
$0={\det}_{ijkk}(h)=(1-q_{ij})h^i_k h^j_k$ 
and hence $h^i_k h^j_k = 0$ for all $i<j$ and all $k$.
\end{proof}

For any $\K$-algebra $A$ upon which $G$ acts via automorphisms, 
the {\em skew group algebra} (sometimes called
the {\em crossed product algebra} or
{\em smash product algebra})
$A\# G$ is the $\K$-vector space $A\ot \KG$ with
multiplication given by 
$$(a\ot g)(b\ot h)=a g(b)\ot g h\ $$
for all  $a,b$ in $A$ and $g,h$ in $G$. 
We write $a t_g$ for $a\ot g$ so that 
the relation in $A\#G$ (or in $\cH_{Q,\kappa}$) is simply
$(at_g)(bt_h)=a\, g(b)\, t_{gh}$.

We may extend the action of $G$ on $V$ to a diagonal action on the tensor algebra $T(V)$ (so that $G$ acts as automorphisms).
Then the algebra $\cH_{Q,\kappa}$ is just the factor algebra
$$
\cH_{Q,\kappa}=T(V)\# G/\langle v_j v_i - q_{ij} v_i v_j - \kappa(v_i,
v_j): 1\leq i,j\leq n \rangle \ ,
$$
where
we write $ab$ for the product $a\otimes b$ in $T(V)$.  
Hence, relation (2) defining $\cH_{Q,\kappa}$ extends to all of $T(V)$.

If $\cH_{Q,\kappa}$ is a \qdha, then 
Proposition \ref{skewpolyalg} implies that 
$G$ acts as automorphisms on the underlying
quantum polynomial algebra $S_Q(V)$, and thus one may form a skew group 
algebra
$S_Q(V)\#G$.
The existence of the standard 
PBW basis here implies that 
the graded algebra associated to $\cH_{Q,\kappa}$ is isomorphic
to $S_Q(V)\# G$.

\section{Noncommutative Gr\"obner Bases Theory}
\label{NoncommutativeTheorySection}

In this section, we recall the use of Gr\"obner bases
in the theory of free associative
algebras. Definitions and formulations used in noncommutative
Gr\"obner bases theory often vary.
Unfortunately, they differ widely
among authors whose work we wish to combine, thus
we give a concise, self-contained account
in this section (and the next) of just those facts necessary for
our main results.
Standard references include~\cite{Green93,Mora,Ufn98}.

Let $\X = \langle x_1, \ldots, x_n\rangle$ be the free monoid in symbols $x_i$.
Its elements are the neutral element (empty word) and nonempty 
words in the alphabet $x_1, \ldots, x_n$ called {\em monomials}.
Let $\KX=\K\langle x_1, \ldots, x_n\rangle$ be the corresponding monoid algebra
over the field $\K$ (i.e., the free associative algebra over $\K$).
We call its elements {\em polynomials}.
Identify the empty word in $\X$ 
with $1$ in $\K$ so that $\KX$
is spanned by monomials as a $\K$-vector space. 
Elements of the form
$x_1^{\alpha_1} x_2^{\alpha_2} \cdots x_n^{\alpha_n}$ with $\alpha_i$ in $\N$ 
are called {\em standard monomials}.

A {\em monomial ordering} on $K\langle X \rangle$ is a
total ordering $\succ$ on $\langle X \rangle$ 
compatible with monomial multiplication
($w u \succ w v$ and $u w \succ v w$
whenever $u\succ v$ for all  $u,v, w$ in $\langle X \rangle$)
that is a well-ordering.
We use the standard definition of {\em leading monomial} $\lm(f)$
and {\em leading coefficient} $\lc(f)$
of a polynomial $f$ in $\KX$.
We say that a monomial $v$ {\em divides} a monomial $w$
if $v$ is a proper subword of $w$, i.e.,
if there exist monomials $m_1,m_2$ in $\X$ such that $w = m_1  v \, m_2$.
For a subset $S \subset \KX$,
the {\em leading ideal} of $S$ is the
two-sided ideal 
$L(S) \; = \; \langle \; \lm(s) \;|\; s \in S\setminus\{0\} \rangle
$ in $\KX$.
Recall that a subset $S \subset I$ is a {\em (two-sided) Gr\"obner basis} 
of the ideal $I$ with respect to $\succ$ if $L(S) = L(I)$. 
In other words, for any nonzero $f$ in $I$, 
there exists $s$ in $S$ with $\lm(s)$ dividing $\lm(f)$.

We are interested in {\em reduced} Gr\"obner bases.
We say that $f$ in $\KX$ is {\em reduced with respect to} $S \subset A$
if no monomial of $f$ is contained in $L(S)$.
A subset $S \subset \KX$ is called {\em reduced}
if for any $s$ in $S$,
$\lm(s)$ does not
divide any monomial of any polynomial from $S$ except $s$ itself.

In Lemma \ref{NFprop}, we see that a monic reduced Gr\"obner 
basis is unique.
We first define a normal form:
\vspace{2ex}

\begin{definition}
Let $\mathcal{S}$ be the set of all ordered subsets of $\KX$ 
and let $\succ$ be a monomial ordering on $\KX$.
A map $\NF : \KX \times \mathcal{S} \to \KX,\; (p,S) \mapsto \NF(p,S)$
is called a {\em normal form} on $\KX$ (with respect to $\succ$)
if for all $f$ in $\KX$ 
and $S$ in $\mathcal{S}$,
\begin{enumerate}
\item[(i)] $\NF(0, S) = 0$,
\item[(ii)] $\NF(f,S) \not= 0 \; \text{implies that} \; \lm\bigl(\NF(f,S)\bigr) \not\in
  L(S)$, and
\item[(iii)] $f - \NF(f,S) \in \langle S\rangle$ .
\end{enumerate}
\end{definition}
\vspace{2ex}

A normal form $\NF$ is called a {\em reduced normal form} if $\NF(f,S)$ is reduced with respect to $S$ for all $f$. A reduced normal form always exists.

\vspace{2ex}

\begin{lemma}
\label{NFprop}
Let $I \subset \KX$ be an ideal, 
$\succ$ a monomial ordering, 
$S \subset I$ a Gr\"obner basis of $I$ with respect to $\succ$, 
and $\NF(\cdot, S)$ a normal form on $\KX$ with respect to $S$ and $\succ$.  
  \begin{enumerate}
  \item[(i)] A polynomial $f$ in $\KX$ lies in $I$ if and only if  
$\NF(f, S) = 0$.
  \item[(ii)] If $J \subset \KX$ is an ideal with $I \subset J$, then $L(I) = L(J)$ implies $I = J$. In particular, $S$ generates $I$ as an ideal of $\KX$.
  \item[(iii)] If $\NF(\cdot , S)$ is a reduced normal form, then it is unique up to
a nonzero constant multiple.
  \end{enumerate}
\end{lemma}
\section{Computation of Gr\"obner Bases}
\label{GBSection}
We now explain how a Gr\"obner basis arises from
an explicit construction of a reduced normal form and illustrate
with group algebras. 
Fix an arbitrary monomial ordering $\succ$ on $\KX$
throughout this section.
\vspace{2ex}

\begin{definition}
\label{overlap}
We say that $f_1$ and $f_2$ in $\KX$ {\em overlap}
if there exist monomials $m_1,m_2$ in $X$ such that
\begin{enumerate}
        \item $\lm(f_1)m_2=m_1 \lm(f_2)$,
        \item $\lm(f_1)$ does not divide $m_1$ and 
$\lm(f_2)$ does not divide $m_2$.
\end{enumerate}
In this case, 
the {\em overlap relation} of $f_1,f_2$ by $m_1,m_2$ is the polynomial
\[
        o(f_1,f_2,m_1,m_2)= \lc(f_2) f_1m_2 - \lc(f_1) m_1f_2.
\]
\end{definition}

\vspace{2ex}

The overlap relation is a generalization of the
$s$-polynomial from the theory of commutative Gr\"obner bases
(see, e.g., \cite{GPS}). 
Note that by construction, 
$$\lm(o(f_1,f_2,m_1,m_2))\prec \lm(f_2)m_2=m_1 \lm(f_2)\, .$$ 
Moreover, there are only finitely many overlaps
between a fixed $f_1$ and $f_2$. Note also that 
a polynomial $f$ can overlap itself.

We define {\em reduction} 
(also called ``inclusion overlap'' or ``spoly'', see~\cite{Mora,Green93}) and the {\em reduction algorithm}, which provides
a desirable coset representative of a polynomial modulo an ideal.
\vspace{2ex}
\begin{definition}
\label{OurReduction}
For any nonzero $f,u$ in $\KX$ with $\lm(u)$ dividing $\lm(f)$, 
define 
$$\NF(f,u):=f- \lc(f)\lc(u)^{-1}\cdot m_1\, u \,m_2\ $$
where
$\lm(f)=m_1\, \lm(u)\, m_2$ for monomials $m_1,m_2$ in $X$.
\end{definition}

\vspace{3ex}

By construction, 
$\lm(\NF(f,u)) \prec \lm(f)$.

\vspace{3ex}

\begin{definition}
Let $S$ be a subset of $\KX$ and fix $f$ in $\KX$.
Define {\em complete reduction of $f$ with respect to} $S$
to be the output
$\NF(f,S)$
of the following procedure $\NF$ applied to $f$ in $\KX$:
\begin{itemize}
\item[(a)]
If $f=0$,  return $f$ and stop.
\item[(b)] If the set $S':=\{ u\in S: \lm(u)$ divides $\lm(f)\}$
is empty, return 
$$\lc(f)\lm(f) + \NF(f-\lc(f)\lm(f),S).
$$
\item[(c)]
Otherwise,
choose some $u$ in $S'$, replace $f$ by $\NF(f,u)$,
and go back to step (a).
\end{itemize}
\end{definition}
\vspace{2ex}

The next two lemmas show that complete reduction defines an algorithm
and that this algorithm is
essentially independent of choices: $f\mapsto \NF(f,S)$ 
is a well-defined function up to a nonzero constant.

\begin{lemma}
The procedure $\NF$ terminates in a finite number of steps. 
\end{lemma}
\begin{proof}
The procedure $\NF$ applied to a nonzero polynomial $f$ produces a 
(non-unique) sequence
of nonzero polynomials $f=f_0, f_1, f_2, \ldots$ 
with strictly decreasing leading monomials:
$\lm(f)\succ\lm(f_1)\succ\lm(f_2)\succ\ldots$.
(Indeed, 
we either apply Step (b)
and set $f_{i+1}= f_{i}-\lc(f_i)\lm(f_i)$
(in order to recursively call $\NF$)
or we apply Step (c) 
and set $f_{i+1}= \NF(f_{i}, u)$ for some monomial $u$.
In either case, $\lm(f_{i-1}) \succ \lm(f_{i})$.)
But $\succ$ is a well-ordering 
and thus the sequence $\lm(f),\lm(f_1),\lm(f_2),\ldots $ is finite.
The procedure thus terminates.
\end{proof}

\begin{lemma}
\label{RedIsNF}
Let $S \subset I$ be a Gr\"obner basis of an ideal $I\subset \KX$.
Then $\NF(\cdot,S)$ is a reduced normal form on $\KX$.
\end{lemma}
\begin{proof}
Recall that $\NF(0,S)=0$ (see Definition~\ref{OurReduction}).
Suppose $h=\NF(f,S)$ for some nonzero polynomial $f$.
Then the reduction algorithm gives
$$
f = \sum_{u\in S} \lc(f)\lc(u)^{-1}\cdot a_u\, u\, b_u + h
$$ 
where $a_u, b_u$  are monomials in $\langle X \rangle$
for each $u$ in $S$. Note 
that $f-h$ lies in $\langle S \rangle$ by construction
and the claim holds for $h=0$.
Now suppose that $h\neq 0$. Then no monomial in $h$ is divisible by $\lm(u)$ for any $u$ in $S$. 
Hence, $\lm(h)\not\in L(S)$ and $h$ is reduced with respect to $S$ as required. Moreover, $\lm(f) = \max_{<}(a_u \lm(u) b_u, \lm(h))$. Thus
$\lm(h) = \lm(f)$ if and only if $f = c\cdot h + g$ for $c\in\K\setminus\{0\}$ and $g\in\langle S\rangle$ with $\lm(g)\prec\lm(f)$. Otherwise $\lm(h) \prec \lm(f)$. 
\end{proof}

The following theorem provides the foundation for the generalized Buchberger's algorithm for the computation of Gr\"obner bases. 
Note that the corresponding algorithm belongs to the family of so-called 
``critical pair and completion'' algorithms (see~\cite{BB85}).

\begin{theorem}[e.g.,~\cite{Gr2}]
\label{DiamondGreen}
Let $S$ be a 
subset of $\KX$. Then
$S$ is a Gr\"obner basis of the ideal $\langle S \rangle$ 
if and only if for any nonzero $f_1,f_2$ in $S$ 
and any overlap relation $o$ of $f_1,f_2$ with 
some monomials $m_1, m_2$ in $X$,
\[
\NF(\; o(f_1,f_2,m_1,m_2), \; S)=0.
\]
\end{theorem}
We will apply this theorem to determine necessary and sufficient conditions
for the set of relations defining $\cH_{Q,\kappa}$ to be a Gr\"obner basis
of the ideal it generates in the appropriate free algebra. 
In the meantime, we illustrate a computation of a Gr\"obner basis on
the group algebra of a finite group.

\begin{proposition}
\label{FinGroup}
Let $G$ be a finite group. Then 
$$
\begin{aligned}
\KG &\cong \K \langle x_g : g\in G \rangle 
/ \langle x_e-1, x_g x_h - x_{gh} : g,h \in G \rangle \\
&\cong \K \langle  x_g : g\in G^* \rangle 
/ \langle S \rangle
\end{aligned}
$$
for $S=\{ x_g x_h - x_{gh}, x_f x_{f^{-1}} - 1 : f, g,h \in G^*,\ gh\not=e \}$.
Let $\succ$ be any monomial ordering on $\K \langle x_g : g\in G^* \rangle$
with $x_{g} x_{h} \succ x_{gh}$ for all $g,h \in G^*$.  
Then $S$ is a reduced Gr\"obner basis with respect to $\succ$
of the ideal $\langle S\rangle$.
\end{proposition}
\begin{proof}
We apply Theorem~\ref{DiamondGreen}.
Consider the polynomial $p = x_g x_h - x_{gh}$ for fixed $g,h\in G^*$. 
Then
$\lm(p) = x_g x_h$ has overlaps with leading monomials 
of the following four types of polynomials from $S$:
\begin{itemize}
\item[(a)] $x_h x_f - x_{hf}$ for any $f\in G^*$; the overlap relation
$$o=(x_g x_h - x_{gh})x_f - x_g (x_h x_f - x_{hf}) = x_g x_{hf} - x_{gh} x_f$$
reduces to $\NF(o)=x_{ghf}-x_{ghf} =0$. 
\item[(b)] $x_f x_g - x_{fg}$ for any $f\in G^*$; the overlap relation 
$$o=x_f (x_g x_h - x_{gh}) - (x_f x_g - x_{fg})x_h  = x_{fg}x_h - x_f x_{gh}$$
reduces to $\NF(o)=x_{fgh}-x_{fgh} =0$. 
\item[(c)] $x_h x_{h^{-1}} - 1$ for any $h\in G^*$; 
the overlap relation reduces to zero as in part (a).
\item[(d)]  $x_{g^{-1}} x_g - 1$ for any $g\in G^*$; the overlap relation 
reduces to zero as in part (b).
\end{itemize}
\end{proof}

Note that there are several modern computer algebra systems 
implementing the theory of noncommutative Gr\"obner bases over
free algebras: \textsc{Bergman}~\cite{Bm},
\textsc{MAGMA}~\cite{Magma},~\textsc{GBNP}~\cite{GBNP} (a package for
\textsc{Gap 4}), \textsc{NCGB}~\cite{NCGB} (a package for
\textsc{Mathematica}, partially written in $C$) and
also \textsc{Singular:Letterplace}~\cite{LL09, LL13}.

\section{Poincar\'e-Birkhoff-Witt Bases}
\label{GBPBWBasisSection}
A natural question arises when working with factor algebras:
What properties must a set of relations exhibit to guarantee a PBW basis?
In this section, we recall how one may establish
a PBW property using Gr\"obner bases and construct a basis
for the associated graded algebra.  We encourage the reader to 
compare with Huishi Li's interesting and well-written
text~\cite{Li} on noncommutative Gr\"obner bases and associated graded
algebras (which appeared in print after this article was completed);
some of the ideas are similar although we are working
in a different context (free algebras over group algebras). 

Let $I$ be an arbitrary ideal in the free algebra $\KX$.
We say that a set $M$ of monomials in $\langle X \rangle$ 
is a {\em monomial $\K$-basis} of 
a factor algebra $\KX/I$
if the cosets $m+I$ for $m$ in $M$ 
form a $\K$-vector space basis of $\KX/I$.
We begin by constructing a monomial $\K$-basis.
\vspace{2ex}

\begin{definition}
Let $I$ be a two-sided ideal of $\KX$
and $\succ$ any monomial ordering on $\KX$.
Define $B_{\succ}$ as the complement
of the leading ideal $L(I)$:
$$
B_{(\succ)}:=
\{\text{monomials }m \in \langle X \rangle : m\notin L(I)\} \ .
$$ 
We call $B_{(\succ)}$ the {\em Gr\"obner coset basis} of $\KX/I$.
\end{definition}
\vspace{2ex}
The term {\em Gr\"obner coset basis} is justified by the following (folklore)
proposition
and the fact that $B_{(\succ)}$ is explicitly constructed from
a Gr\"obner basis.

\vspace{1ex}
\begin{prop}
\label{GroebnerCosetBasis}
Let $I$ be a two-sided ideal of $\KX$ and let
$\succ$ be any monomial ordering on $\KX$.
Then $B_{(\succ)}$ is a monomial $\K$-basis of $\KX/I$.
\end{prop}
\begin{proof}
Let $B\subset \langle X \rangle$ be any set of monomials.  Since
$L(I)$ is a monomial ideal,
$B+L(I)$ is a $\K$-basis of
$\KX /L(I)$ if and only if
$B=B_{(\succ)}$.
Any $a$ in $\KX$ is equivalent to the normal form $\NF(a,S)$
modulo $I$, where $S$ is a reduced Gr\"obner basis of $I$.
But since $S$ and $\NF$ are reduced,
every $\NF(a,S)$ lies in $\text{Span}_{\K}B_{(\succ)}$ 
by definition.
Hence, $B_{(\succ)}$ spans $\KX/ I$ as a $\K$-vector space.
The set $B_{(\succ)}$ is also $\K$-independent modulo I:
If any finite linear combination of monomials in $B_{(\succ)}$
would lie in $I$, then its leading monomial
would lie in $L(I)\cap B_{(\succ)}=\emptyset$.
\end{proof}

\vspace{2ex}

Gr\"obner technology allows one to 
describe the explicit shape of relations lending themselves to a 
$\K$-basis of standard monomials.
\begin{prop}
\label{UseGrobner}
Let $I$ be a two-sided ideal of $\KX$.
Suppose there exists a monomial ordering $\succ$
with respect to which 
$I$ has reduced Gr\"obner basis $S$ of the form
\[
S =\{x_j x_i - p_{ij}: 1\leq i<j\leq n 
\}\ 
\]
for some $p_{ij}$ in $\KX$ with $x_j x_i \succ \lm(p_{ij})$
for each $i<j$. Then the factor algebra
$\KX/I$ has monomial basis
$\{x_1^{\alpha_1} \ldots x_n^{\alpha_n} \mid \alpha_i\in\N \}$.
\end{prop}
\begin{proof}
Let $S$ be a reduced Gr\"obner basis 
of $I$ with respect to any monomial ordering $\succ$.
Then the leading ideal $L(S)$ consists of all
non-standard monomials 
if and only if 
$L(S)$ is generated by $x_j x_i$ for $ 1\leq i<j\leq n$
(since $L(S)$ is a monomial ideal).
As $S$ is reduced, 
this is equivalent to $S= \{ x_j x_i - \ p_{ij} : \ x_j x_i \succ \lm(p_{ij})
 \}$. Thus,
$B_{\succ}$ 
is the set of standard monomials
if and only if $S$ has the given form.
The result then follows from
Proposition~\ref{GroebnerCosetBasis}.
\end{proof}


The last proposition gives an immediate proof of the well-known fact
that quantum polynomial algebras satisfy a PBW property.
\vspace{2ex}

\begin{cor}
\label{KQV}
Let $S = \{ v_j v_i - q_{ij} v_i v_j:1\leq i< j \leq n \} \subset \K\langle v_1, \ldots, v_n \rangle$. Then for any monomial ordering $\succ$ on $\K\langle v_1, \ldots, v_n \rangle$, \ $S$ is a Gr\"obner basis of $\langle S \rangle$.
Hence 
$\{v_1^{\alpha_1} \ldots v_n^{\alpha_n} \mid \alpha_i\in\N \}$
is a monomial $\K$-basis
of $S_Q(V) = \K\langle v_1, \ldots, v_n \rangle / \langle S \rangle$.
\end{cor}

\quad

\begin{example}
\label{exPBW}
Consider $A= \K\langle x,y \rangle/\langle xy-y^2 \rangle$. 
Suppose $\succ$ is any
monomial ordering of $\K\langle x,y\rangle$ with
$x\succ y$.  Then $xy\succ y^2$
and $\{xy-y^2\}$ is a Gr\"obner basis of the
ideal $I$ it generates with respect to $\succ$. 
The Gr\"obner coset basis, 
$B_{(\succ)}=\{ y^a x^b: a,b\in \N \}$,  
is a monomial $\K$-basis of $\K\langle x,y\rangle/I$ 
as Proposition~\ref{UseGrobner} implies. On the other hand, 
the set of standard
monomials $\{x^a y^b:a,b\in \N\}$ does not form a 
monomial $\K$-basis of $\K\langle x,y\rangle/I$ since, e.g., 
$xy + \langle xy-y^2 \rangle = -y^2 + \langle xy-y^2 \rangle$.

Now consider instead a monomial ordering $>$
on $\K\langle x,y\rangle$ with $y>x$. Then
$y^2>xy$ and 
the Gr\"obner basis $S$ of $\langle -y^2 + xy \rangle$
with respect to 
$>$ is an infinite set, $S=\{ y x^n y - x^{n+1}y : n\in \N \}$ 
(see~\cite{Ufn98}).
Notice that $x^2,xy,yx$ all lie in the Gr\"obner coset basis
$B_{(>)}$, as they do not lie in the ideal of leading monomials of $S$.
By Proposition~\ref{GroebnerCosetBasis},
the Gr\"obner coset basis $B_{(>)}$ is a monomial $\K$-basis of 
$\K\langle x,y\rangle/I$,
yet it is not a Poincar\'e-Birkhoff-Witt basis 
(as it contains $x^2, xy, yx$).
Note that $\{xy-y^2\}$ is not a Gr\"obner basis 
of the ideal it generates with respect to $>$.
\end{example}

\vspace{2ex}

The Poincar\'e-Birkhoff-Witt theorem for universal enveloping algebras
of Lie algebras
has several possible analogs in the setting of finitely presented
associative $\K$-algebras.
Applying a fixed permutation
to the indices in the set of standard monomials 
$x_1^{\alpha_1} \ldots x_m^{\alpha_m}$ 
may yield a monomial $\K$-basis for $\KX/I$ 
for some permutations but not others, as we saw in the last
example.
We appeal to the associated graded algebra.
Let $A=\KX/I$ be an arbitrary factor algebra (with $I$ a two-sided
ideal in $\KX$). 
Let $\cA=\{A_i : i\geq -1 \}$ be 
an ascending $\N$-filtration of $A$.
Note that any $\N$-filtration on $\KX$ (for example, by degree)
induces an $\N$-filtration on a factor algebra of $\KX$.
Recall that the 
associated graded algebra $\Gr^{\cA}(A)$ of $A$ with respect to the
filtration $\cA$ is  
$$
\Gr(A)=\Gr^{\cA}(A)=\bigoplus_{i\in\N}\ A_{i}\,/\, A_{i-1}\ .
$$
One may choose any $\K$-vector space (direct sum) complement to $A_{i-1}$
in $A_i$ to obtain a vector space isomorphism,
$A\cong \Gr(A)$.

We say that any $f$ 
in $\KX$ has $\cA$-{\em degree} 
$d\geq 0$ in $\N$
whenever $f+I\in A_d$ but $f+I\notin A_{d-1}$
and we write $\deg_{A}(f)=d$ in this case.
Set $\deg_{\cA} (f) = -\infty$ for any $f$ in $I$.
We call a monomial ordering $\succ$ on $\KX$ {\em compatible} with 
the filtration $\cA$ if 
$$ 
\deg_{\cA}(f)>\deg_{\cA}(f')
\quad\text{ implies}\quad
\lm(f)\succ \lm(f')
$$ 
for all $f,f'$ in $\KX$.
Note that many compatible monomial orderings
exist for a fixed $\N$-filtration on $A$. 
We say a set $M$ of monomials in $\KX$ is a 
{\em monomial $\K$-basis} of the associated graded 
algebra $\Gr(A)$ if the elements
$m + I+A_{\deg_{\cA}(m)-1}$ for $m$ in $M$
form a $\K$-basis of $\Gr(A)$, and we record a straightforward
observation.

\vspace{1ex}

\begin{prop}\label{GradedToNonGradedAlgV}
Let $\KX/I$ be an $\N$-filtered algebra. 
Then
\begin{itemize}
\item[(1)]
Any monomial $\K$-basis of
$\Gr(\KX/I)$ is also a monomial $\K$-basis
of $\KX/I$.
\item[(2)]
The set $B_{(\succ)}$ is a monomial $\K$-basis
for both $\KX/I$ and $\Gr(\KX/I)$, for
any monomial ordering $\succ$ compatible with 
the $\N$-filtration.
\end{itemize}
\end{prop}
\begin{proof}
One may check directly that the set of $m+I$ for $m$ in a 
monomial $\K$-basis of $\Gr(A)$
spans $A=\KX/I$ and is linearly independent.
Now suppose some nonzero, finite
$\K$-linear combination
of monomials $m_i$ in $B_{(\succ)}$
has degree $d$ with respect to the filtration.
Then the compatibility
of $\succ$ and the filtration force each $\deg(m_i)\leq d$.
By Proposition~\ref{GroebnerCosetBasis}, $B_{(\succ)}$
is a monomial $\K$-basis of $A$.
Hence for each $d$, the set
$\{m+I: m\in B_{(\succ)},\ \deg(m)\leq d\}$ spans $A_{d}$
and $\{m+I+A_{d}: m\in B_{(\succ)},\ \deg(m)=d\}$ spans $A_d/A_{d-1}$
over $\K$.
This set is also $\K$-linearly independent:
If any nonzero, finite linear combination of monomials in $B_{(<)}$
of $\deg d$ defined the zero class in $A_d/A_{d-1}$, 
the degrees of all the monomials in the combination
would be $d-1$ instead of $d$.  Thus $B_{(\succ)}$
is also a $\K$-monomial basis for $\Gr(A)$.
\end{proof}


In the special case that our factor algebra is
$\cH=\cH_{Q,\kappa}$, we may relate the PBW property of the original
algebra to that of the quantum polynomial algebra.
Formally, we filter the free associative algebra
$$
\FF=\K\langle v_1,\ldots,v_n,t_g : g\in G^* \rangle
=\bigoplus_{i=0}^{\infty}\, \FF_i$$ 
by 
assigning degree $0$ to all $t_g$ (for $g$ in $G$)
and degree $1$ to all $v$ in $V$
and consider the associated graded algebra
$$\Gr\,\cH
:=\oplus_{i=0}^n \ \cH_i/\cH_{i-1}\ ,
$$
where $\cH_i$ is the image of $\FF_i$
under the projection $\FF \rightarrow \cH$.
Assuming that $Q$ is a quantum system of parameters,
the graded algebra $\Gr\cH$ is isomorphic to a quotient
of the quantum polynomial ring $S_Q(V)\# G$, and 
$\cH$ has the standard PBW basis
if and only if $\Gr\cH$ and 
$S_Q(V)\# G$ are isomorphic (as graded algebras).
In fact, Naidu and Witherspoon~\cite{NW} observe that
every quantum Drinfeld Hecke algebra is isomorphic to a formal
deformation of $S_Q(V)\# G$.

We end this section by recording a few other 
facts about \qdha s.
\begin{theorem}
If $\cH_{Q,\kappa}$ is a \qdha, then
\begin{enumerate}
\item[(i)] $\cH_{Q,\kappa}$ is Noetherian;
\item[(ii)] $\cH_{Q,\kappa}$ is an integral domain if and only if $G$ is trivial;
\item[(iii)] the Gel'fand-Kirillov dimension of $\cH_{Q,\kappa}$ is 
$$\gkdim \cH_{Q,\kappa} = n + \gkdim \K\ ;$$
\item[(iv)] if $|G|$ is not divisible by $\Char \K$, then the global homological dimension of $\cH_{Q,\kappa}$ is at most $n$.
\end{enumerate}
\end{theorem}
\begin{proof}
${}_{}$
\begin{enumerate}
\item[(i)] Since $S_Q(V)$ is Noetherian (e.g., see~\cite{BGV}), so is $S_Q(V)\# G$
(see~\cite[Proposition~1.6]{Passman}).
Then as
$\Gr \cH_{Q,\kappa}\cong S_Q(V) \# G$, the filtered algebra
$\cH_{Q,\kappa}$ is as well
(see, e.g.,~\cite{MR}).
\item[(ii)] In $\K G$,
$0=1 - (t_g)^d = (1-t_g)(1+t_g + \ldots + t_{g^{d-1}})$
for any $g\in G$ of order $d>1$.
\item[(iii)] Consider the filtration $\{\cH_k:{k\geq -1}\}$ of $\cH=\cH_{Q,\kappa}$. 
Let $d=|G|$.  Then 
${\limsup}_{k \ra \infty} \log_k (k^n \frac{d}{n!} + \ldots) =  n$
as the PBW property implies that
$$\dim_{\K} \cH_k = {k+n \choose n} \cdot d = k^n \frac{d}{n!} + 
\text{ (lower order terms)}.$$ 
\item[(iv)] In the non-modular case (see~\cite[Theorem~7.5.6]{MR}),
$$
\gldim(S_Q(V) \# G) = \gldim(S_Q(V)) = n\, .
$$ 
Then 
$\gldim(\cH_{Q,\kappa}) \leq \gldim(\Gr \cH_{Q,\kappa}) = n$
(by~\cite[Theorem~7.6.18]{MR})
since $\Gr \cH_{Q,\kappa}\cong S_Q(V) \# G$.
\end{enumerate}
\vspace{-2ex}
\end{proof}

\begin{remark}
One might ask about the possibility of grading $\cH_{Q,\kappa}$ directly.
The group algebra
$\K G$ is graded if and only if the graded degree (weight) of each
$t_g$ is 0. 
The relation $t_g v_k = g(v_k) t_g$ in $\cH_{Q,\kappa}$ is graded
if all $v_i$ have the same weight, say 1. 
But the relation $v_j v_i = q_{ij} v_i v_j + \kappa(v_i, v_j)$ 
is graded only in two cases: Either the weight of every $v_k$ is zero 
or each $\kappa(v_i, v_j)$ is zero, 
since otherwise the graded degree of 
$\kappa(v_i, v_j)$ is zero while the graded degree 
of $v_j v_i - q_{ij} v_i v_j$ is 2. 
In the first case, $\cH_{Q,\kappa}$ is trivially graded
(all weights zero).
In the second case, $\cH_{Q,\kappa}=S_Q(V)\#G$.
\end{remark}

\section{Conditions on Parameters}
\label{ConditionsOnParametersSection}
In this section, we deploy the theory of Gr\"obner bases to rigorously
establish necessary and sufficient conditions for $\cH_{Q,\kappa}$
to define a~\qdha. 
We write the factor algebra $\cH_{Q,\kappa}$ as $\FF/ \left<{R'}\right>$,
where $\FF$ is the free associative $\K$-algebra
\begin{equation}\label{F}
\FF=\K\langle v_1,\ldots,v_n,t_g : g\in G^* \rangle\ 
\end{equation}
and
$\left<{R'}\right>$ is the ideal in $\FF$
generated by relations defining $\cH_{Q,\kappa}$,
\begin{equation}\label{relations}
\begin{split}
{R'}=\{ t_g t_h - t_{gh},\
t_g v_i - g(v_i) t_g,\
v_j v_i -& q_{ij} v_i v_j - \kappa(v_i, v_j):\\
&\text{ for all } g,h\in G^*, 1\leq i,j\leq n
\} .
\end{split}
\end{equation}
Moreover, let us define the smaller set of relations
\begin{equation}\label{relationsRed}
\begin{split}
R=\{ t_g t_h - t_{gh},\ t_g v_i - g(v_i) t_g,\
v_j v_i -& q_{ij} v_i v_j - \kappa(v_i, v_j):\\
&\text{ for all } g,h\in G^*, 1\leq i<j\leq n
\} .
\end{split}
\end{equation}
Before expressing the PBW property of $\cH_{Q,\kappa}$
in terms of a Gr\"obner basis, 
we must ensure that the given monomial ordering 
is compatible.
\vspace{2ex}
\begin{definition}
\label{useOrdering}
Consider a monomial ordering $\succ$ on 
the free algebra
$\FF$ 
which satisfies $v_1 \succ \ldots \succ v_n \succ t_g $ for all $g\in G^*$.
We say that $\succ$
{\em preserves the rewriting procedure} of relations of $\cH_{Q,\kappa}$
if
\begin{itemize}
\item $t_g t_h \succ t_{gh}$ for all $g,h \in G^*$, 
\item $v_j v_i \succ t_{g}$ for all $i,j$ and $g\in G^*$,
\item $v_j v_i \succ v_i v_j$ for all $i<j$ (``first misordering preference''), and
\item $t_g v_i \succ v_j t_g$ for all $i,j$ and $g\in G^*$ (``second misordering preference'').
\end{itemize}
\end{definition}

\vspace{2ex}

\begin{remark}
\label{rewritingexists}
A monomial ordering which preserves the rewriting procedure
always exists.
One example can be constructed as follows.
We assign degree 1 to each $t_g$ for $g$ in $G^*$ and to
each $v_i$ for $1\leq i\leq n$.
Two monomials in $\FF$ are first compared by their total degree. 
In the case of equal degrees, misordering preferences are applied.
If two monomials are of the same total degree and no misordering preference can be applied, 
we compare monomials further with left lexicographical ordering.
\end{remark}

\vspace{2ex}

\begin{prop}
\label{GB2PBW}
Suppose $\succ$ is any monomial ordering on $\FF$ with
$v_1\succ\ldots \succ v_n \succ t_g$
which preserves the rewriting procedure
of $\cH_{Q,\kappa}$.
If $R$ is a Gr\"obner basis of $\langle R\rangle$ with respect to $\succ$,
then
$\FF/ \langle R \rangle$ has 
monomial $\K$-basis
$$B=\{
v_1^{\alpha_1}\cdots v_n^{\alpha_n} t_g: 
\alpha \in \N^n, g\in G \}\ .$$
\end{prop}
\begin{proof}
The set of leading monomials of $R$ in 
$\FF$
is 
\[
L := \{v_j v_i, \ t_g v_i, \  t_g t_h \ \mid g,h\in G^*, 1\leq i<j \leq n \},
\]
and 
$B_{(\succ)}=\langle X\rangle\setminus (\langle X\rangle\cap\langle L \rangle) = B$.
Thus if $R$ is a Gr\"obner basis of $\langle R \rangle$,
then $\langle L\rangle=L(R)=L(\langle R \rangle)$ and $B$ is a 
monomial $\K$-basis of
$\FF/\langle R \rangle$ by 
Proposition~\ref{GroebnerCosetBasis}.
\end{proof}


We now give conditions for $R$ to be a Gr\"obner basis.
\begin{theorem}\label{formalassociativityconditions}
Let $\succ$ be any monomial ordering on $\FF$ with 
$v_1 \succ\ldots\succ v_n \succ t_g$
that preserves the rewriting procedure of $\cH_{Q,\kappa}$.
Then $R$ is a Gr\"obner basis of $\langle R \rangle$
with respect to $\succ$ if and only if
for all $g,h$ in $G$ and
$1\leq i< j< k\leq n$,
\begin{enumerate}
\item[(i)]
\rule{0ex}{3ex}
${}_{}$

\vspace{-2.5ex}

$
\begin{aligned}
0 \ = \ \   
&\bigl(q_{ik} q_{jk}\, h v_k - v_k\bigr)\ \kappa_h(v_i,v_j) 
+\,  \bigl(q_{jk} v_j - q_{ij}\, h v_j\bigr)\ \kappa_h(v_i,v_k)\\
& +\,  \bigl(h v_i - q_{ij} q_{ik} v_i\bigr)\ \kappa_h(v_j,v_k)\, ,
\end{aligned}
$
\item[(ii)]
\ \ \ \
$g(v_j) g(v_i) = q_{ij} g(v_i) g(v_j)\ ,$ \   and
\rule{0ex}{4ex}
\item[(iii)]
\ \ \ \ 
$\kappa_{h^{-1}gh}(v_i,v_j)  = \ \displaystyle{
\sum_{k < \ell}} \rule{0ex}{4ex}
\ {\det}_{ijkl}(h)\, \kappa_g(v_k,v_\ell)\, .$
\end{enumerate}
Moreover, if $R$ is a Gr\"obner basis, it is reduced.
\end{theorem}
\begin{proof}
We derive necessary and sufficient conditions under which
$R$ is a Gr\"obner basis of the
ideal it generates in the free associative algebra $\FF$
using Theorem~\ref{DiamondGreen}: We
examine all overlap polynomials $o=o(f_1,f_2,m_1,m_2)$
with $f_1,f_2$ in $R$ and $m_1,m_2$ monomials in $\FF$.
Setting the complete reduction $\NF(o,R)$ of each overlap
to zero in $\FF$
gives a set of necessary and sufficient conditions for $R$
to be a Gr\"obner basis of the ideal $\langle R \rangle$ in $\FF$.

By Lemmas~\ref{RedIsNF} and~\ref{NFprop},
the algorithm $\NF$ produces a reduced normal form
and hence its output is unique up to a nonzero constant.  
Thus the algorithm $\NF$
gives a result independent (up to a nonzero
scalar) of any choices in the algorithm.
We forgo the explicit computations and just record the results here.

Since the set of relations of the group algebra
$\K G$ forms a Gr\"obner basis of the ideal it generates
in the free algebra
$\K\langle t_g \mid g \in G^* \rangle$
(by Proposition~\ref{FinGroup}, for example), we are left
with only three kinds of possibly nonzero
overlaps between elements from $R$:

(a) There is an overlap relation between 
$t_h t_g-t_{hg}$ and $t_g v-g(v)t_g$
for any $v=v_i$ and $g,h$ in $G$,
namely,
$$o=(t_h t_g-t_{hg})v-t_h(t_g v-g(v)t_g)\ .
$$
The complete reduction algorithm applied to $o= -t_{hg}v+t_h g(v) t_g$
yields zero: 
%
$\NF(o,R)=0$ for this type of overlap.

(b) There is an overlap relation between elements
$v_k v_j-q_{jk}v_j v_k - \kappa(v_j, v_k)$ and 
$v_j v_i-q_{ij}v_i v_j - \kappa(v_i, v_j)$ 
for distinct $1\leq i,j,k\leq n$
obtained by multiplying the first on the right by $v_i$
and the second on the left by $v_k$.
Applying the complete reduction algorithm
gives $\NF(o,R)$
as the \textit{non-degeneracy expression} (see~\cite{NDC})
\[
\sum_g 
(q_{ik} q_{jk} g(v_k) - v_k) a^{ij}_g t_g
+
(q_{jk} v_j - q_{ij} g(v_j)) a^{ik}_g t_g
+
(g(v_i) - q_{ij} q_{ik} v_i) a^{jk}_g t_g
\]
(where we abbreviate $a_g^{ij}:=\kappa_g(v_i,v_j)$),
which is zero in $\FF$
if and only if
$$
\begin{aligned}
0\ = \ 
(q_{ik} q_{jk} g v_k - v_k)\ \kappa_g(v_i,v_j)
 +\ (q_{jk} v_j - & q_{ij} g v_j)\ \kappa_g(v_i,v_k)\\
+\ & (g v_i - q(i,j) q_{ik} v_i)\ \kappa_g(v_j,v_k) \ 
\end{aligned}
$$
for each $g$ in $G$.
This is precisely condition (i) of the theorem.

(c) For all $h$ in $G$ and $i\neq j$,
there is an overlap relation $o$
between $t_h v_j-h(v_j)t_h$
and $v_j v_i-q_{ij}v_i v_j - \kappa(v_i, v_j)$ obtained
by multiplying the first on the right by $v_i$
and the second on the left by $t_h$:
$$
\begin{aligned}
o & =\ 
t_h (v_j v_i - q_{ij}v_iv_j - \kappa(v_i,v_j)) - (t_h v_j - h(v_j) t_h ) v_i \\
& =\   - q_{ij}t_h v_i v_j - t_h \kappa(v_i,v_j)  + h(v_j) t_h v_i\ .
\end{aligned}
$$
The complete reduction algorithm reduces $o$ 
to
\begin{equation}\label{thirdformalrelation2}
\begin{aligned}
\NF(o,R)\, = 
 \,  & \sum_{k < \ell} 
 \biggl(
q_{k\ell}\ {\det}_{ijk\ell}(h) + {\det}_{ij\ell k}(h)
\biggr)\ v_{k} v_{\ell}\, t_h + 
\sum_{k} {\det}_{ijkk}\ v_{k}^2\, t_h \\
& +
\sum_{g\in G} 
\Bigg(
\sum_{k < \ell}\ {\det}_{ijk\ell}(h)\   
\kappa_g(v_k,v_\ell)\ 
- \kappa_{h^{-1}gh}(v_i,v_j)\Bigg) t_{gh}  \, .\\
\end{aligned}
\end{equation}
But $\NF(o,R)$ vanishes in $\FF$
exactly when the coefficient of each monomial in 
Equation~\ref{thirdformalrelation2} vanishes.
Lemma~\ref{auts} implies that
the coefficient of each
$v_k v_{\ell} t_{gh}$  and each 
$v_m^2 t_{gh}$ vanish in Equation~\ref{thirdformalrelation2} if and only if 
condition (ii) of the theorem holds.
The coefficient of each $t_{gh}$ in Equation~\ref{thirdformalrelation2} 
vanishes for all $i\neq j$ exactly when 
condition (iii) of the theorem holds.
\end{proof}

\vspace{3ex}

\begin{remark}
\label{distinct}
Observe that if $\cH_{Q,\kappa}$ has the standard PBW basis,
then conditions (i), (ii), (iii) for $i<j<k$
in the above theorem are equivalent 
to conditions (i), (ii), (iii) with {\em arbitrary} indices $i,j,k$.
Indeed, from the definition of quantum minor:
$${\det}_{jikl}(h)=-q_{ji}\,{\det}_{ijkl}(h)$$
for all $h$ in $G$.  If $\cH_{Q,\kappa}$ has the standard PBW basis, then
Proposition~\ref{skewpolyalg} implies
in addition that $\kappa(v_j, v_i)=q_{ij}^{-1}\kappa(v_i, v_j)$.
These two facts allow us to replace increasing by arbitrary indices in the conditions
of the theorem when it might be helpful.
\end{remark}

We now use the last theorem and the connection between Gr\"obner
bases and standard bases in the last section to show that 
$\cH_{Q,\kappa}= \FF/ \langle R' \rangle$ 
has the standard PBW basis 
if and only if the conditions of the last theorem hold.
(We will see in Section~\ref{HochschildCohomologySection} that these conditions
have a natural interpretation in terms of Hochschild cocycles.)

\vspace{2ex}

\begin{theorem}
\label{formalassociativityconditions2}
The factor algebra
$\cH_{Q,\kappa}$ is a \qdha\
if and only if the following four conditions hold:
\begin{enumerate}
\item[(i)] 
The matrix $Q$ 
is a quantum system of parameters
and $G$ acts on the
quantum polynomial algebra $S_Q(V)$
as automorphisms,
\item[(ii)]
\rule{0ex}{2ex}
The parameter $\kappa$ defines a quantum 2-form:
$$
\kappa(v_i, v_j)= - q_{ji}^{-1} \kappa(v_j, v_i) \text{ for distinct  }i,j\ ,
$$
\item[(iii)]
\rule{0ex}{2ex}
For all $h$ in $G$ and $1\leq i<j<k\leq n$,
\begin{align*}
& 0 \ =\ \ \bigl(q_{ik} q_{jk}\, h v_k - v_k\bigr)\ \kappa_h(v_i,v_j) \\
& \quad\quad + 
\bigl(q_{jk} v_j - q_{ij}\, h v_j\bigr)\ \kappa_h(v_i,v_k)\\
& \quad\quad +
\bigl(h v_i - q_{ij} q_{ik} v_i\bigr)\ \kappa_h(v_j,v_k)\ ,
\end{align*}
\item[(iv)]
\rule{0ex}{.5ex}
For all $g,h$ in $G$ and all $1\leq i<j\leq n$,
$$
\kappa_{h^{-1}gh}(v_i,v_j)  = \sum_{k < \ell} {\det}_{ijkl}(h)\, \kappa_g(v_k,v_\ell)\ .
$$
\end{enumerate}
\end{theorem}
\begin{proof}
Fix any monomial ordering $\succ$ on $\FF$ which satisfies
the rewriting procedure with $v_1\succ v_2\succ \ldots v_n\succ t_g$
(see Remark~\ref{rewritingexists} for an explicit choice).
By Theorem~\ref{formalassociativityconditions},
the conditions of the theorem imply that
$R$ is a Gr\"obner basis of the ideal it generates
and that $\cH_{Q,\kappa}=\FF/\langle R' \rangle
=\FF/\langle R\rangle$.  Thus $\cH_{Q,\kappa}$
has the standard PBW basis
by Proposition~\ref{GB2PBW}.

Conversely, assume that $\cH_{Q,\kappa}$ has the standard PBW basis.
Proposition~\ref{skewpolyalg} implies conditions (i) and (ii)
and thus $\cH_{Q,\kappa}=\FF/\langle R' \rangle=\FF/\langle R\rangle$.  
We saw in the proof Theorem~\ref{formalassociativityconditions}
that the overlap polynomial $o$ of any elements in $R$
has normal form $\NF(o)$ lying in $\text{span}_{\K}(B)$.
But each $\NF(o)$ lies in $\langle R \rangle$ as well (since each overlap $o$ does).
Thus each $\NF(o)$ gives a linear dependence modulo $\langle R \rangle$
among elements of $B$.  
As $B$ is a standard PBW basis, 
each $\NF(o)$ must then be zero in the free
algebra $\FF$.  
Thus $R$ is a Gr\"obner basis of 
the ideal it generates
by Theorem~\ref{DiamondGreen}.
The result then follows from
Theorem~\ref{formalassociativityconditions}.
\end{proof}

The last theorem immediately implies (set $\kappa\equiv 0$)
the following corollary.
\begin{cor}\label{PBWforquantumalgebra}
Suppose $G$ acts as automorphisms on a quantum polynomial algebra
$S_Q(V)$.  Then 
$B=\{v_1^{\alpha_1} \ldots v_n^{\alpha_n} t_g \mid \alpha_i\in\N, \ g\in G\}$
is a monomial $\K$-basis for
$S_Q(V)\# G$.
\end{cor}

\vspace{2ex}

\begin{remark}
Fix some $q$ in $\K$ and 
suppose $q_{ij}=q$ for
all $1\leq i<j\leq n$.
Then for all $i,j,k$, the first part of condition (iii) 
of the last theorem is equivalent to
\[0\ = \
(q^2 g v_k - v_k) \kappa_g(v_i,v_j) + q(v_j - g v_j) \kappa_g(v_i,v_k)+ (g v_i - q^2 v_i) \kappa_g(v_j,v_k) \ .
\]
\end{remark}

\vspace{3ex}

\begin{remark} 
Condition (iii) from the last theorem
can be written explicitly 
in terms of the entries of any matrix
$h$ in $G$.  Again, fix scalars $h_b^a$ in $\K$ with $h(v_a)=\sum_b h_b^a v_b$
and abbreviate $a^{ij}_g$ for $\kappa_g(v_i,v_j)$.
Then condition (iii) holds exactly when
\begin{itemize}
\item 
$0=(q_{ik} q_{jk} h^k_k -1 )a^{ij}_g - q_{ij} h^j_k a^{ik}_g + h^i_k a^{jk}_g \ ,$
\item 
$0=q_{ik} q_{jk} h^k_j a^{ij}_g + (q_{jk} - q_{ij} h^j_j) a^{ik}_g + h^i_j a^{jk}_g\ ,$
\item 
$0=q_{ik} q_{jk} h^k_i a^{ij}_g - q_{ij} h^j_i a^{ik}_g + (h^i_i - q_{ij} q_{ik}) a^{jk}_g\ ,$
\quad\quad\text{and}
\item
$0=q_{ik} q_{jk} h^k_\ell a^{ij}_g - q_{ij} h^j_\ell a^{ik}_g + h^i_\ell a^{jk}_g \ 
$
\end{itemize}
for all $g$ in $G$, all $i<j<k$, and
any $\ell$ not in $\{i,j,k\}$.
\end{remark}

\vspace{1ex}


\section{Abelian Groups}\label{AbelianSection}

In this section, we assume $G$ in $\GL_n(\K)$
is abelian acting diagonally on $v_1, \ldots, v_n$.
Let $\chi_i:G\rightarrow \K^*$ be the linear character 
recording the $i$-th diagonal entry, i.e.,
$gv_i = \chi_i(g) v_i$ for all $g$ in $G$ and $1\leq i\leq n$.
We deform the skew group algebra $S_Q(V)\# G$ by setting each $q$-commutator
$v_i v_j-q_{ji} v_j v_i$ to a group element $g$ whose $i$-th and $j$-th entries
are inverse and whose $k$-th entry is the scalar that arises upon interchanging $v_i v_j$ and $v_k$ in the quantum algebra $S_Q(V)\# G$:
$$
(v_i v_j) v_k = (q_{ki} q_{kj})\ v_k (v_i v_j).
$$
In fact, we will take linear combinations of such group elements $g$ and also
insist
that {\em every} element in $G$ has inverse $i,j$ entries:
$\chi_i=\chi_j^{-1}$.
Indeed, we apply
Theorem~\ref{formalassociativityconditions2}
carefully for diagonal actions to deduce the following corollary.

\vspace{1ex}

\begin{cor}
Suppose $G$ is abelian acting diagonally.
Then $\cH_{Q,\kappa}$ is a \qdha\
if and only if the following hold:
\begin{itemize}
\item
 $Q$ is a quantum system of parameters,
\item
$\kappa$ is a quantum $2$-form,
\item
for all $g$ in $G$ and
$i\neq j$,
$\kappa_g(v_i,v_j)\neq 0$ implies that
$\chi_i=\chi^{-1}_j$ and 
$\chi_k(g) =q_{ki}\, q_{kj}$ for all $k \neq i,j$\ .
\end{itemize}
\end{cor}
The next proposition 
gives a complete description of \qdha\ in the abelian setting.

\begin{prop}
Suppose $G$ is an abelian group acting diagonally
on the basis $v_1, \ldots,v_n$.
Then the set of \qdha s comprises all factor algebras
of the form
$$
\K\langle v_1, \ldots, v_n\rangle \# G /
\big\langle
v_j v_i - q_{ij} v_i v_j -
\sum_{g \in G^{ij}} c_g^{\, ij}\ g:
\ 1\leq i < j \leq n \ \text{with}\
\chi_i=\chi_j^{-1}
\big\rangle
$$
where 
$G^{ij}=\{g\in G:
\chi_k(g)=q_{ki}q_{kj}
\ \text{for all}\ k\neq i,j\}
$ and
the $c_g^{\, ij}$ are arbitrary scalars in $\K$.
\end{prop} 

\vspace{0ex}

\begin{remark}
Suppose $\cH_{Q,\kappa}$ is a \qdha\
with $G$ acting diagonally
on $v_1, \ldots, v_n$. 
Fix $q$ in $\K$ and suppose $q_{ij}=q$ for all $i<j$.
Then if $c_g^{\, ji}\neq 0$ in the last corollary,
$g$ is a diagonal matrix
$$g=\diag(q^2,\ldots, q^2, c, 1,\ldots, 1, c^{-1}, q^{-2}, \ldots, q^{-2})
$$
with entries $c$ and $c^{-1}$ at $i$-th and $j$-th locations, respectively,
for some $c\in\K$.
\end{remark}

\vspace{0ex}
\begin{example}
Let $q$ be a primitive odd $k$-th 
root-of-unity in $\K$ and let $G$ be the group 
of order $k$ generated by the diagonal matrix
$$
g=\text{diag}
(q^2, 1 , q^{-2})\, .
$$
Then the $\K$-algebra $\cH$ generated by
symbols $v_1,v_2,v_3,t_{g}$
with relations
$$
\begin{aligned}
\hphantom{xxx}
& 
t_{g}^k=1, \
 t_{g} v_1=q^{2}\,v_1 t_{g},
\
t_{g} v_2=v_2 t_{g},
\
t_{g} v_3=q^{-2}\,v_3 t_{g},\\
& v_2v_1=q\,v_1 v_2, \
v_3v_2=q\,v_2 v_3, \
v_3 v_1=q\,v_1 v_3 + \sum_{i=1}^k c_i t_{g}^i\, ,
\end{aligned}
$$
for arbitrary constants $c_i$ in $\K$, is a \qdha.
\end{example}

\quad

See Naidu and Witherspoon~\cite{NW} for an explicit description
of the related Hochschild cohomology (and the cocycles
defining these algebras) for groups acting diagonally
in characteristic zero.


\section{Hochschild Cohomology}\label{HochschildCohomologySection}

Using results of Naidu and Witherspoon~\cite{NW}, one may interpret
the conditions of Theorem~\ref{formalassociativityconditions2} 
in terms of the Hochschild cohomology
of the associated skew group algebra, $\HHD(S_Q(V)\#G)$.
We assume $\K=\mathbb{C}$ (the complex numbers)
in this section and 
fix a quantum system of parameters  $Q=(q_{ij}\mid 1\leq i,j\leq n)$ defining a 
quantum polynomial algebra $S_Q(V)$.
Recall that Hochschild cohomology is a generalization
of group cohomology to a bimodule setting:
For a $\K$-algebra $C$,
$\HHD(C)=\Ext^{\DOT}_{C^e}(C, C)$ 
where $C^{e}$ is the enveloping algebra $C\otimes C^{op}$.

We may regard the quantum exterior algebra $\bigwedge_Q(V)$
(see Example~\ref{quantumexterioralgebra})
as a factor algebra of a 
quantum polynomial algebra
with respect to a nearly opposite set of scalars:
 $$\bigwedge_Q(V)=S_{Q'}(V)/ \langle v_1^2, \ldots, v_n^2 \rangle$$ 
where $Q'=(q_{ij}'\, |\, 1\leq i,j\leq n)$ is the quantum system of parameters
with $q_{ij}'=-q_{ij}$ for $i\neq j$
and $q_{ii}'=1$ for each $i$.
Proposition~\ref{auts} (together with Corollary~\ref{transpose})
applied to $S_{Q'}(V)$ then easily implies the following two observations
(where $h^t$ is the transpose of $h$).
\begin{cor}
The group $G$ acts as automorphisms on $\bigwedge_Q(V)$
if and only if for all  $h$ in $G$,
\begin{itemize}
\item[(i)]
${\det}_{ijk \ell}(h)= q_{\ell k}\ {\det}_{ij\ell k}(h)\ 
\quad\text{ for all } 1\leq i,j,k,\ell \leq n\, ,
\text{ and }$
\item[(ii)]
$\det_{ijkk}(h^t)=0\ \ \,
\quad\quad\quad\quad \ \quad\text{ for all } k \text{ and }i<j\, .$ 
\end{itemize}
\end{cor}

\begin{cor}
$G$ acts on as automorphisms on {\em both} 
$S_Q(V)$ and on $\bigwedge_Q[V]$ if and only if
for all $h$ in $G$,
\begin{itemize}
\item[(i)]
$\det_{ijkl}(h)=0\ \ \ \text{  for all } i, j, k, l$, and
\item[(ii)]
$\det_{ijkk}(h^t)=0 \ \ \text{ for all } k \text{ and }i<j\, .$
\end{itemize}
\end{cor}

As in Shepler and Witherspoon~\cite{SW1, SWbracket} (in the nonquantum setting),
Naidu and Witherspoon recommend associating a Hochschild cocycle to
the parameters $Q,\kappa$ defining a factor algebra $\cH_{Q,\kappa}$.
Any quantum 2-form $\kappa$ 
(see Proposition~\ref{skewpolyalg})
extends to an element of
$$
\Hom_\K\Big(\bigwedge^2_Q V,\ S_Q(V)\#G\Big)
\cong
\Hom_{S_Q(V)^e}\Big(S_Q(V)^e\otimes\bigwedge^2_QV, \ S_Q(V)\#G\Big)\, ,
$$
and thus defines a 2-cochain in the theory of Hochschild cohomology
$$
\HHD(S_Q(V),\ S_Q(V)\#G)\ 
$$
computed using a quantum Koszul resolution on $S_Q(V)$
(see~\cite{NW}).
But (see~\cite[Theorem~3.5]{NW})
$$
\HHD(S_Q(V),\ S_Q(V)\#G)^G\cong \HHD(S_Q(V)\# G)\ .
$$ 
Thus, one wonders: When does $\kappa$ define
a class in the Hochschild cohomology
$\HHD(S_Q(V)\#G)$, the
cohomology theory detecting all algebraic deformations
of $S_Q(V)\# G$ ?
Results of Naidu and Witherspoon~\cite{NW} imply
the following proposition.
\begin{prop}\label{Hochschildrelations}
Assume $G$ acts as automorphisms on both $\bigwedge_Q(V)$
and $S_Q(V)$ and $\kappa$ is a quantum $2$-form. Then
\begin{itemize}
\item
Condition (iii) of Theorem~\ref{formalassociativityconditions2}
 holds if and only if $\kappa$ is a cocycle.
\item
Condition (iv) of Theorem~\ref{formalassociativityconditions2}
 holds if and only if $\kappa$ is invariant.
\end{itemize}
\end{prop}
Theorem~\ref{formalassociativityconditions2} and
Proposition~\ref{Hochschildrelations} together with
Theorem~3.5 of Naidu and Witherspoon~\cite{NW}
therefore give
another interpretation of the necessary and sufficient
PBW conditions: 
\begin{theorem}
Assume $G$ acts on both the quantum polynomial algebra
$S_Q(V)$ and 
the quantum exterior algebra
$\bigwedge_Q(V)$ as automorphisms.
Let $\kappa$ be a quantum $2$-form.
Then the factor algebra $\cH_{Q,\kappa}$ is a \qdha\ if and only if
$\kappa$ induces
a Hochschild cocycle
for $S_Q(V)\# G$ .
\end{theorem}
Naidu and Witherspoon~\cite[Theorem~4.4]{NW} in fact 
show that every ``constant''
Hochschild 2-cocycle gives rise to a \qdha\
(extending a theorem from the nonquantum setting;
see~\cite{SW1}).

\section{Automorphisms of Coordinate Rings of Quantum Planes}
\label{QuantumPlanesSection}

In this section, we consider automorphisms of quantum polynomial
algebras and \qdha s over any 2-dimensional vector space $V$.
Recall that every \qdha\ $\cH_{\kappa, Q}$ arises from a group acting 
as automorphisms on some quantum polynomial algebra $S_Q(V)$
(by Proposition~\ref{skewpolyalg}).
Every graded $\K$-automorphism of 
a quantum polynomial algebra $S_Q(V)$
restricts to a linear map on $V$
and thus defines an element of $\GL_{n}(\K)$.
Conversely, a transformation in $\GL_{n}(\K)$
extends to a graded $\K$-automorphism
of $S_Q(V)$ when it satisfies the condition of Lemma~\ref{auts}.

We write $q$ for the parameter $q_{12}$.
Recall that the monomial matrices in $\GL_2(\K)$ 
are simply those which are either diagonal or anti-diagonal.
For $n=2$, it is not difficult
to determine the group $\Aut_{\K} S_Q(V)$ of
graded $\K$-automorphisms of $S_Q(V)$ explicitly
(see, e.g., Alev-Chamarie~\cite{AlevChamarie}):

\begin{proposition} 
If $n=2$, then $\Aut_{\K} S_Q(V)$ is
\begin{itemize}
\item $\GL_2(\K)$ when $q=1$, 
\item $(\K^*)^2$ (the torus) when $q\neq \pm 1$, and
\item the subgroup of monomial matrices of $\GL_2(\K)$ when $q=-1$.
\end{itemize}
\end{proposition}

We describe the set of \qdha\ in each of the above three cases
by applying Theorem~\ref{formalassociativityconditions2}.

\vspace{2ex}

\begin{remark}\label{quantumspeciallinear}
Condition (iv) of Theorem~\ref{formalassociativityconditions2}
for $n=2$ implies that for any commuting
$g$ and $h$ in $G$,
$\kappa_g(v_1, v_2)={\det}_Q(h)\, \kappa_g(v_1, v_2)\, ,$
where $\det_Q$ is the {\em quantum determinant} 
defined by
$${\det}_Q
\begin{pmatrix}
a & b\\
c& d
\end{pmatrix}
:=ad-q\, bc\ .
$$ 

\noindent 
Thus, for any\ \qdha\ $\cH_{\kappa, Q}$ and 
for any $g$ in $G$, 
the parameter $\kappa_g$ is identically zero 
unless 
the centralizer subgroup $Z_G(g)$ of $g$ in $G$
lies in the set of quantum-determinant-one matrices,
$$\{M \in \GL_2(\K):{\det}_Q(M)=1\} \, . $$
In particular, every \qdha\ is supported on group elements of quantum 
determinant one. 
\end{remark}

\vspace{2ex}

\subsection{Coordinate Ring of Nonquantum Plane} {($n=2,q=1$)}
${}_{}$

\vspace{1ex}

\noindent
When $q=1$,
the set of \qdha s
comprises all quotients of the form
$$
\K\langle x,y\rangle\# G/
\Big\langle\, xy-yx-\sum_{\substack{g\in G\\ \det(g)=1}} c_g t_g\,\Big\rangle
$$
where the scalars $c_g$ in $\K$ are arbitrary for
$g$ in a set
of determinant-one conjugacy class representatives of $G\leq\GL_2(\K)$
and $c_{h^{-1}gh}=\det(h)\, c_g$ for all $h$ in $G$.
Note that the coefficient $c_g$ is zero 
or the centralizer $Z_G(g)$ is a subgroup of $\SL_2(\K)$.
(In particular, the coefficient of the identity group element
is zero unless $G\leq\SL_2(\K)$.)
These nonquantum algebras are called
{\em graded Hecke algebras} (see~\cite{EtingofGinzburg}
and~\cite{RamShepler}, for example).
(In fact, Remark~\ref{quantumspeciallinear} 
is an quantum analogue
of an aspect of the characteristic zero theory of graded Hecke algebras.)

\vspace{2ex}

\subsection{Coordinate Ring of Transcendental Quantum Plane}
${}_{}$

\noindent
{($n=2,q \neq \pm 1$)}

\vspace{1ex}

\noindent
Quantum Drinfeld Hecke algebras in 2 dimensions for $q\neq \pm 1$ 
(including the case of $q$ transcendental over a subfield)
all arise
from an abelian group $G$ acting diagonally and are described
in Section~\ref{AbelianSection}.  
If each element of $G$ has quantum determinant 1
($\det_Q(g)=1$ for all $g$ in $G$), 
then the set of \qdha s comprises all quotients
of the form
$$
\K\langle x,y\rangle\# G/
\big\langle\, xy-q\, yx-\sum_{g\in G} c_g t_g\,\big\rangle
$$
where the scalars $c_g$ in $\K$ are arbitrary.
If some element of $G$ has non-unity quantum determinant,
then $\kappa$ is identically zero
(by Remark~\ref{quantumspeciallinear}), and $\cH_{\kappa,Q}$
is just the quantum polynomial algebra $S_Q(V)=S_q(V)$ on two variables.

\vspace{2ex}

\subsection{Coordinate Ring of Skew Quantum Plane} {($n=2,q=-1$)}
${}_{}$

\vspace{1ex}

\noindent
The set of \qdha s in 2 dimensions when $q=-1$ 
comprises all quotients of the form
$$
\K\langle x,y\rangle\# G/
\big\langle\, xy+ yx-\sum_{g\in G} c_g t_g\,\big\rangle
$$
where the scalars $c_g$ in $\K$ are arbitrary for
$g$ in a set
of conjugacy class representatives of a monomial group $G\leq \GL_2(\K)$
and $c_{h^{-1}gh}=\det_Q(h) \, c_g$ for all $h$ in $G$. In particular,
$c_g=0$ if some element $h$ of the centralizer
$Z_G(g)$ has non-unity quantum determinant ($\det_Q(h)\neq 1$).

\section{Automorphisms of the Coordinate Ring of Quantum $3$-space}
\label{QuantumSpaceSection}

Various authors examine
automorphisms and graded automorphisms of quantum polynomial algebras
and their generalizations (for example,
see Kirkman, Kuzmanovich, and Zhang~\cite{KirkmanKuzmanovichZhang},
Alev and Chamarie~\cite{AlevChamarie},
and Artamonov and Wisbauer~\cite{ArtamonovWisbauer}).
The group $(\K^*)^n$ of diagonal matrices is always a subgroup
of the group of graded automorphisms, $\Aut_{\K} S_Q(V)$, of $S_Q(V)$.
When the parameters $q_{ij}$ are independent over $\K^*$,
$\Aut_{\K} S_Q(V)$ contains no other automorphisms.
For arbitrary parameters, the situation is more complicated to describe.
In this section,  
we give $\Aut_{\K}S_Q(V)$ for $n=3$ explicitly.
A careful analysis of Lemma~\ref{auts}
for $n=3$ (with help from the computer
algebra system \textsc{Singular}~\cite{Singular}) 
leads to the following theorem, whose
proof we omit for the sake of brevity. 
\begin{theorem}\label{autosin3dims}
Let $k$ be a field and $\K = k(q_{12}, q_{13}, q_{23})$ an extension.
Consider the coordinate ring of the quantum affine 3-space
\[
S_Q(V) = 
\
\K\langle v_1,v_2,v_3 \mid v_2 v_1 = q_{12} v_1 v_2, \
v_1 v_3 = q_{31} v_3 v_1, \
v_3 v_2 = q_{23} v_2 v_3 \rangle.
\]
Then $\Aut_{\K} S_Q(V)$
is exactly one of the following groups: 
\begin{itemize}
\item[(i)] 
\rule[0ex]{0ex}{4ex}
If all $q_{ij}=1$, then $\Aut_{\K} S_Q(V) = \GL_3(k)$. 
(Here, $\K=k$ and $\trdeg_k \K = 0$.) 
\item[(ii)]
\rule[0ex]{0ex}{4ex}
If all $q_{ij}=-1$, then
$\Aut_{\K} S_Q(V)$ 
is the subgroup 
 of monomial matrices in $\GL_3(k)$.
(See Corollary~\ref{monomialgroup}.) Also, $\K=k$ and $\trdeg_k \K = 0$. 
\item[(iii)]
\rule[0ex]{0ex}{4ex}
$\Aut_{\K} S_Q(V)=(\K^*)^3$ and $\trdeg_k \K \leq 3$ unless
\begin{itemize}
\item[$\bullet$]
$q_{12}=q_{23}=q_{31}$, or 
\item[$\bullet$]
$\{q_{12}, q_{23}, q_{31}\}=\{\pm 1, c, c^{-1}\}$ 
for some $c$ in $\K^*$.
\end{itemize}
\item[(iv)] 
\rule[0ex]{0ex}{4ex}
If $q_{12}=q_{23}=q_{31} \neq \pm 1$,
then 
$\trdeg_k \K \leq 1$ and
$\Aut_{\K}S_Q(V)$ is generated by 
\[ 
\Bigg\{
 \left[
 \begin{array}{*{3}{c}}
 0 & a_{12} & 0 \\
 0 & 0 & a_{23} \\
 a_{31} & 0 & 0
 \end{array}
 \right], 
 \left[
 \begin{array}{*{3}{c}}
 0 & 0 & a_{13} \\
 a_{21} & 0 & 0 \\
 0 & a_{32} & 0
 \end{array}
 \right]
\Bigg\}
\subset  \GL_n(\K)\ .
\]
\item[(v)]
\rule[0ex]{0ex}{4ex}
 If $\{q_{12}, q_{23}, q_{31} \}= \{1, c, c^{-1}\}$ for some $c\neq 1 $, then
$\trdeg_k \K \leq 1$
\footnote{We give upper bounds for $\trdeg$, allowing further evaluation of quantum parameters $q_{ij}$ in addition to the given conditions on them.} 
and three cases arise:
  \begin{itemize}
  \item[(a)] If $q_{23}=1$ and $q_{12}=q_{31}^{-1}\neq 1$, then
$$
\hphantom{xxxxxxxxxxxxxxx}
\Aut_{\K}S_Q(V) =
\Bigg\{
\left[
\begin{array}{*{3}{c}}
a_{11} & 0 & 0 \\
0 & a_{22} & a_{23} \\
0 & a_{32} & a_{33}
\end{array}
\right] 
\Bigg\}
\leq \GL_n(\K)
.$$
\item[(b)] If
$q_{31}=1$ and $q_{12}=q_{23}^{-1}\neq 1$, then
$$
\hphantom{xxxxxxxxxxxxxxx}
\Aut_{\K}S_Q(V)=
\Bigg\{
\left[
\begin{array}{*{3}{c}}
a_{11} & 0 & a_{13} \\
0 & a_{22} & 0 \\
a_{31} & 0 & a_{33}
\end{array}
\right]  
\Bigg\}
\leq \GL_n(\K)
.$$
\item[(c)] If $q_{12}=1$ and $q_{23}=q_{31}^{-1}\neq 1$, then
$$
\hphantom{xxxxxxxxxxxxxxx}
\Aut_{\K}S_Q(V)=
\Bigg\{
\left[
\begin{array}{*{3}{c}}
a_{11} & a_{12} & 0 \\
a_{21} & a_{22} & 0 \\
0 & 0 & a_{33}
\end{array}
\right] 
\Bigg\}
\leq \GL_n(\K)
.$$
\end{itemize}
\item[(vi)]
\rule[0ex]{0ex}{4ex}
If 
$\{q_{12}, q_{23}, q_{31}\}=\{-1, c, c^{-1}\}$ 
for some $c\neq -1$ in $\K^*$,
then $\trdeg_k \K \leq 1$
and $\Aut_{\K}S_Q(V)$ is generated by 
$(\K^*)^3$ together with
$$
\Bigg\{
\left[
\begin{array}{*{3}{c}}
a_{11} & 0 & 0 \\
0 & 0 & a_{23} \\
0 & a_{32} & 0
\end{array}
\right] 
\Bigg\}
\subset \GL_n(\K)
\text{ if } q_{23}=-1, q_{12}=q_{31}^{-1}\ ,
$$
$$
\Bigg\{
\left[
\begin{array}{*{3}{c}}
0 & 0 & a_{13} \\
0 & a_{22} & 0 \\
a_{31} & 0 & 0
\end{array}
\right] 
\Bigg\}
\subset \GL_n(\K)
\text{ if }q_{31}=-1, q_{12} = q_{23}^{-1}\, ,
$$
or
$$
\Bigg\{
\left[
\begin{array}{*{3}{c}}
0 & a_{12} & 0 \\
a_{21} & 0 & 0 \\
0 & 0 & a_{33}
\end{array}
\right] 
\Bigg\}
\subset \GL_n(\K)
\text{ if } 
q_{12}=-1, q_{31}^{-1}=q_{23}\ .
$$
\end{itemize}
\end{theorem}
In the next section, we will give an example using this theorem.
\section{Example}\label{ExtensiveExampleSection}

In this section, we show how to use our results to work
out the complete set of \qdha s arising from a fixed
group.  We assume the characteristic of $\K$ is not two in this example.

Consider the subgroup $G$ of $\GL_3(\K)$
generated by the two matrices
\[
M = 
\begin{footnotesize}
\left[
\begin{array}{crc}
0 & 0 & 1 \\
0 & -1 & 0 \\
1 & 0 & 0
\end{array}
\right]
\end{footnotesize}
\text{ and }\
N =
\begin{footnotesize}
\left[
\begin{array}{crr}
1 & 0 & 0 \\
0 & -1 & 0 \\
0 & 0 & -1
\end{array}
\right],
\end{footnotesize}
\]
and note that $G$ is isomorphic to the dihedral group $D_8$ of order $8$.
Set $g_1=e,g_2=M,g_3=N$, 
$$
\begin{footnotesize}
g_4 = MN = 
\left[
\begin{array}{ccr}
0 & 0 & -1 \\
0 & 1 & 0 \\
1 & 0 & 0
\end{array}
\right], \quad
\end{footnotesize}
\begin{footnotesize}
g_5 = NM = 
\left[
\begin{array}{rcc}
0 & 0 & 1 \\
0 & 1 & 0 \\
-1 & 0 & 0
\end{array}
\right], 
\end{footnotesize}
$$
$$
\begin{footnotesize}
g_6 = MNM = 
\left[
\begin{array}{rrc}
-1 & 0 & 0 \\
0 & -1 & 0 \\
0 & 0 & 1
\end{array}
\right]
\end{footnotesize},\quad
\begin{footnotesize}
g_7 = NMN = 
\left[
\begin{array}{*{3}{c}}
0 & 0 & -1 \\
0 & -1 & 0 \\
-1 & 0 & 0
\end{array}
\right],
\end{footnotesize}
\quad \text{and}
$$
$$
\begin{footnotesize}
g_8 = MNMN =
\left[
\begin{array}{rcr}
-1 & 0 & 0 \\
0 & 1 & 0 \\
0 & 0 & -1
\end{array}
\right].
\end{footnotesize}
$$
We use Theorems~\ref{formalassociativityconditions2}
and~\ref{autosin3dims}
and the computer
algebra system \textsc{Singular}~\cite{Singular}
to determine parameters $q_{ij}$ and $\kappa(v_i,v_j)$ such that
$\cH_{Q,\kappa}$ is a \qdha.
Condition (i) is satisfied when $q_{12}q_{23}=1$ and $q_{13}= \pm 1$.
Conditions (ii), (iii), and (iv) provide us with a linear system in terms of 
the $\kappa_g(v_i,v_j)$. We abbreviate notation and write 
$\kappa_k(i,j)$ for $\kappa_{g_k}(v_i,v_j)$.
Computing minimal associated prime ideals from a primary decomposition in the affine space of parameters, we arrive at
all possibilities yielding a factor algebra
$\cH:=\cH_{Q,\kappa}$ which satisfies the PBW property.
The following relations $\mathcal{R}$ define
all quantum Drinfeld Hecke algebras 
$\cH \cong \K\langle v_1,v_2,v_3 \rangle\# G/ 
\langle \mathcal{R} \rangle $. 
\hspace{-10ex}
\begin{enumerate}
\item[]
\hspace{-8ex}
{\bf (I)} \ For $q_{13}=1, q_{12}q_{23}=1$: 
\begin{enumerate}
\item[] 
\rule{0ex}{3ex}
\hspace{-8ex}
(a)\ \ 
If $q_{12}\neq q_{23}$, then the relations are
\[
v_2 v_1 = q_{12} v_1 v_2, \  v_3 v_2 = q_{12}^{-1} v_2 v_3, \ v_3 v_1 = v_1 v_3\, .
\]
\item[]
\rule[0ex]{0ex}{0ex}
\hspace{-7ex}
(b)\ \ 
If $q_{12} = q_{23}$,
then the parameter $\kappa_4(1,3)$ can be chosen freely 
in $\K$ and the relations are
\[
\phantom{12345678910}
v_2 v_1 = q_{12} v_1 v_2, \  v_3 v_2 = q_{12} v_2 v_3, \ v_3 v_1 = v_1 v_3 +
\kappa_4(1,3) (t_{g_4} - t_{g_5})\, .
\]
\end{enumerate}
\item[]
\hspace{-8ex}
{\bf (II)} \ For $q_{13}=-1, q_{12}q_{23}=1$: 
\begin{enumerate}
\item[] 
\hspace{-7ex}
(c) \ 
\rule[0ex]{0ex}{3ex}
If $q_{12}^2=-1$ (giving a primitive fourth-root-of-unity),
then $\kappa_2(1,3)$ can be chosen freely in $\K$ and 
\[\hphantom{xxxxxxxxx}
v_2 v_1 = q_{12} v_1 v_2, \ v_3 v_2 = - q_{12} v_2 v_3, \ v_3 v_1 = - v_1 v_3 + \kappa_2(1,3) (t_{g_2} - t_{g_7})\, .
\]
\item[] 
\hspace{-7ex}
(d) \
\rule[0ex]{0ex}{0ex}
Otherwise, the relations are
$$
\hphantom{xxxxxxxx}
v_2 v_1 = q_{12} v_1 v_2, \ v_3 v_2 = q_{12}^{-1} v_2 v_3,  \ v_3 v_1 =- v_1 v_3\, .
$$
\end{enumerate}
\end{enumerate}

Note that in the nonquantum setting, when $q_{13}=q_{12}=q_{23}=1$,
we recover a one-parameter family of classical Hecke Drinfeld algebras
from Case~(I)(b).  In the quantum setting, 
we obtain several other one-parameter families of algebras.

\bibliographystyle{amsplain}
\bibliography{qjac}
\end{document}